\documentclass[11pt,a4paper,reqno]{amsart}
\usepackage[english]{babel}
\usepackage[applemac]{inputenc}
\usepackage[T1]{fontenc}
\usepackage{palatino}
\usepackage{amsmath}
\usepackage{amssymb}
\usepackage{amsthm}
\usepackage{amsfonts}
\usepackage{graphicx}
\usepackage{color}
\usepackage{mathtools}

\usepackage[colorlinks = true, citecolor = black]{hyperref}
\title[A Koebe distortion theorem for quasiconformal mappings in the Heisenberg group]{A Koebe distortion theorem for quasiconformal mappings\\ in the Heisenberg group}
\author[T.\ Adamowicz]{Tomasz Adamowicz}
\address{The Institute of Mathematics, Polish Academy of Sciences \\ ul. \'Sniadeckich 8, 00-656 Warsaw, Poland}
\email{tadamowi@impan.pl}
\author[K. F\"assler]{Katrin F\"assler}
\address{Department of Mathematics\\ University of Fribourg \\ Chemin du Mus\'{e}e 23,
CH-1700 Fribourg, Switzerland}
\email{katrin.faessler@unifr.ch}
\author[B.\ Warhurst]{Ben Warhurst}
\address{Institute of Mathematics, University of Warsaw\\ ul.Banacha 2, 02-097 Warsaw, Poland}
\email{B.Warhurst@mimuw.edu.pl}

\thanks{K.F.\ was supported by the Swiss National Science Foundation through the grant 161299 `Intrinsic rectifiability and mapping theory on the Heisenberg group'.}
\date{\today}

\newcommand{\R}{\mathbb{R}}

\newcommand{\C}{\mathbb{C}}

\newcommand{\diam}{\operatorname{diam}}

\newcommand{\dist}{\operatorname{dist}}

\newcommand{\Om}{\Omega}
\newcommand{\Hei}{{\mathbb{H}}^{1}}

\newcommand{\finv}{{f^{-1}}}

\def\Barint_#1{\mathchoice
          {\mathop{\vrule width 6pt height 3 pt depth -2.5pt
                  \kern -8pt \intop}\nolimits_{#1}}%
          {\mathop{\vrule width 5pt height 3 pt depth -2.6pt
                  \kern -6pt \intop}\nolimits_{#1}}%
          {\mathop{\vrule width 5pt height 3 pt depth -2.6pt
                  \kern -6pt \intop}\nolimits_{#1}}%
          {\mathop{\vrule width 5pt height 3 pt depth -2.6pt
                  \kern -6pt \intop}\nolimits_{#1}}}

\numberwithin{equation}{section}

\theoremstyle{plain}
\newtheorem{thm}[equation]{Theorem}

\newtheorem{lemma}[equation]{Lemma}

\newtheorem{cor}[equation]{Corollary}
\newtheorem{proposition}[equation]{Proposition}

\theoremstyle{definition}

\newtheorem{definition}[equation]{Definition}

\theoremstyle{remark}
\newtheorem{remark}[equation]{Remark}

\definecolor{blau}{rgb}{0.1,0.0,0.9}
\newcommand{\blue}{\color{blau}}
\definecolor{lila}{rgb}{0.5,0.0,0.5}

\definecolor{Bcolor}{rgb}{0.5,0.0,0.0}

\newcounter{komcounter}
\numberwithin{komcounter}{section}

\newcommand{\kom}[1]{}
\renewcommand{\kom}[1]{{\bf \blue /#1/}}

\usepackage{hyperref}% http://ctan.org/pkg/hyperref
\hypersetup{%
  colorlinks = true,
  linkcolor  = black
}

\begin{document}

\begin{abstract}
We prove a Koebe distortion theorem for the average derivative of a quasiconformal mapping between domains in the sub-Riemannian Heisenberg group $\Hei$.
Several auxiliary properties of quasiconformal mappings between subdomains of
$\Hei$ are proven, including BMO-estimates for the logarithm of the Jacobian.
Applications of the Koebe theorem include diameter bounds for images of curves, comparison of integrals of the average derivative and the operator norm of the horizontal differential, as well as the study of quasiconformal densities and metrics in domains in $\Hei$. The theorems are discussed for the sub-Riemannian and the Kor\'anyi distances. This extends results due to Astala--Gehring, Astala--Koskela, Koskela and Bonk--Koskela--Rohde.
\end{abstract}

\maketitle
\tableofcontents
\vfill\eject
\section{Introduction}

The Koebe distortion theorem is a classical result in complex analysis that provides control over the absolute value of the derivative of a conformal function between domains in the complex plane \cite[Corollary 1.4]{pommerenke1975univalent}, see also \cite[Theorem 1.6]{MR777305}. K.\ Astala and F.\ Gehring \cite[Theorem 1.8]{MR777305} extended this result to the class of quasiconformal maps in $\mathbb{R}^n$, $n\geq 2$.

\begin{thm}[Astala, Gehring]\label{t:AstGehr} Let $n \geq 2$. For every $K\geq 1$, there exists a constant $1\leq c_K<\infty$ such that for every $K$-quasiconformal map $f:\Omega \to \Omega'$  between domains in
 $\mathbb{R}^n$ with $\Om\subsetneq \mathbb{R}^n$, it holds
\begin{equation}\label{eq:KoebeEst}
\frac{1}{c_K} \frac{d(f(x),\partial \Omega')}{d(x,\partial \Omega)}\leq a_f(x) \leq c_K \frac{d(f(x),\partial \Omega')}{d(x,\partial \Omega)}\quad \text{for all } x\in \Omega.
\end{equation}
Here,
\begin{displaymath}
a_f(x) := \exp \left(\frac{1}{n}\frac{1}{\mathcal{L}^n(B(x))}\int_{B(x)} \log J_f \;\mathrm{d}\mathcal{L}^n\right),\quad B(x):=B(x,d(x,\partial \Omega)),
\end{displaymath}
and $\mathcal{L}^n$ denotes Lebesgue measure on $\mathbb{R}^n$.
\end{thm}
 Quasiconformal mappings are not necessarily differentiable everywhere, but they belong to the Sobolev class $W_{loc}^{1,n}$. Consequently, Theorem \ref{t:AstGehr} is formulated not for the pointwise derivative, but for $a_f$. This is a natural geometric quantity which, for $n=2$ and $f$ conformal, agrees with $|f'(z)|$. Both $a_f$ and Theorem \ref{t:AstGehr}  have found various applications, for instance in connection with the global distortion properties of quasiconformal mappings \cite{MR1191747}, diameter bounds for
images of curves \cite{Koskela1994}, in the studies of conformal metrics \cite{bkr98}, and more recently related to harmonic quasiconformal mappings \cite{MR3430458}. We address counterparts of some of these results as well as their generalizations.

More precisely, the goal of this paper is to prove a Koebe distortion theorem for quasiconformal mappings in the Heisenberg group and to study several applications thereof. The Heisenberg group $\mathbb{H}^1$ endowed with a left-invariant sub-Riemannian metric $d_s$ has played an important role as a testing ground and motivational example for the extension of the theory of quasiconformal maps from Euclidean to more abstract metric spaces. This development can be seen from a series of papers and notes \cite{Pansu1, MR1317384, zbMATH00746069, HeinonenNotes, MR1654771}. Given the role of the sub-Riemannian Heisenberg group in the development of the theory of quasiconformality, and the wealth of quasiconformal mappings which can be constructed in this particular space by methods described in \cite{MR1317384,zbMATH00784029,zbMATH01640258,zbMATH06195912,Aus}, we consider $\mathbb{H}^1$ a natural non-Euclidean setting where it is worthwhile to study counterparts for $a_f$ and Koebe's theorem.

\begin{definition}\label{d:a_f}
For a quasiconformal map $f:\Omega \to \Omega'$ between domains $\Omega,\Omega'  \subsetneq \mathbb{H}^1$, we define
\begin{equation}\label{eq:def_a_f}
a_f(x):= \exp \left(\tfrac{1}{4} \left(\log J_f\right)_{B(x)}\right)
\end{equation}
with $B(x):=B(x,d(x,\partial \Omega))$ and $u_B:= \frac{1}{m(B)}\int_B u \;\mathrm{d}m$.
\end{definition}
Here and in the following, $B=B(x,r)$ denotes an open ball with center $x$ and radius $r>0$ with respect to a metric $d$ which will depend on the context. Moreover, $\lambda B:= B(x,\lambda r)$.
 %and $d(x,\partial \Omega):= \inf\{d(x,y):\; y\in \partial \Omega\}$.
 The measure $m$ is a bi-invariant Haar measure on $\Hei$ as defined in Section \ref{ss:Heis}. A \emph{domain} is an open connected set. The constant $4$ which appears in \eqref{eq:def_a_f} is unrelated to the factor $4$ in Koebe's distortion theorem for conformal functions in the plane, but instead agrees with the Hausdorff dimension of the sub-Riemannian Heisenberg group. The following is the main theorem of this paper.

\begin{thm}\label{t:KoebeHeis}
For every $K\geq 1$, there exists a constant $1\leq c_K<\infty$ such that for every $K$-quasiconformal mapping $f:\Omega \to \Omega'$ between domains in $\mathbb{H}^1$ with $\Om\subsetneq \Hei$, it holds
\begin{equation}\label{eq:KoebeEst}
\frac{1}{c_K} \frac{d(f(x),\partial \Omega')}{d(x,\partial \Omega)}\leq a_f(x) \leq c_K \frac{d(f(x),\partial \Omega')}{d(x,\partial \Omega)}\quad \text{for all } x\in \Omega.
\end{equation}
\end{thm}

Theorem~\ref{t:KoebeHeis} is flexible with respect to the choice of the underlying distance in $\Hei$.
In the Heisenberg group one often considers two bi-Lipschitz equivalent distances: the \emph{sub-Rie\-mannian distance} $d_s$ and the \emph{Kor\'{a}nyi distance} $d_{\mathbb{H}^1}$, see Section \ref{ss:Heis} for the definitions. Our results apply both to $d=d_s$ and $d=d_{\mathbb{H}^1}$. Since the two distances are bi-Lipschitz equivalent, a homeomorphism $f:\Omega \to \Omega'$ is quasiconformal with respect to $d_{\mathbb{H}^1}$ if and only if it is quasiconformal with respect to $d_s$. More is true: as explained in \cite[\S 1.1]{MR1317384}, one obtains the same class of $K$-quasiconformal mappings, $K\geq 1$, with respect to either metric.
The definition of $a_f$ as given in Definition~\ref{d:a_f} depends on the metric $d$ used to define the ball $B(x)=B(x,d(x,\partial \Omega))$; let us momentarily denote $a_f^{\mathbb{H}^1}$ and $a_f^s$ to indicate dependence on $d_{\mathbb{H}^1}$ or $d_s$, respectively. Using Theorem \ref{t:BMO_BMO_loc}, Theorem \ref{c:logBMOdomain}, and \eqref{e:comp_BMO_diff_metric}, we deduce
by a similar argument as in the proof of Lemma \ref{l:a_f_comparison}
 that for every $K\geq 1$, there exists a constant $0<\Lambda_K<\infty$ such that
\begin{displaymath}
\Lambda _K^{-1} a_f^{\mathbb{H}^1}(x) \leq a_f^{s}(x)\leq \Lambda_K a_f^{\mathbb{H}^1}(x),\quad\text{for all }x\in\Omega.
\end{displaymath}
It follows that once we have established Theorem \ref{t:KoebeHeis} for either the Kor\'{a}nyi or the sub-Riemannian distance, then it also holds for the other one.\smallskip

\noindent \textbf{Proof and applications of the main result.}
A crucial ingredient in the proof of Theorem \ref{t:KoebeHeis} is the following result, which we establish both with respect to the Kor\'{a}nyi distance $d_{\Hei}$ and the sub-Riemannian distance $d_s$. The necessary concepts, in particular BMO spaces and BMO seminorms $\|\cdot\|_{\ast}$ on open sets $\Omega \subset \Hei$, are introduced in Section \ref{s:JacobianQC}.
\begin{thm}\label{c:logBMOdomain}
Let $f:\Omega \to \Omega'$ be a $K$-quasiconformal map between domains in $\Hei$. Then $\log J_f$ belongs to $\mathrm{BMO}(\Omega)$ with a bound for $\|\log J_f\|_{\ast}$ in terms of $K$.
\end{thm}
As far as we know, a direct proof of this result in the case $\Om$ is a domain, not the whole space, does not appear explicitly in the literature, even in the Euclidean setting (cf.\ \cite[Remark 2]{MR0361067}). One way to obtain the result is by first proving that $\log J_f$ belongs to some \emph{local} BMO space $\mathrm{BMO}_{loc}(\Omega)$ and then using
 the identity $\mathrm{BMO}(\Omega)=\mathrm{BMO}_{loc}(\Omega)$. This is the approach which we pursue here. In the case of the sub-Riemannian distance $d_s$, the equality of $\mathrm{BMO}$ and $\mathrm{BMO}_{loc}$ goes back to work of S.\ Buckley and O.\ Maasalo \cite{MR1749313,maas}. We employ results by S.\ Staples \cite{staples} in order to deduce the corresponding identity for the Kor\'{a}nyi distance $d_{\mathbb{H}^1}$ in place of $d_s$. To be precise, Staples' result is used to establish the following.
 \begin{thm}\label{t:introThm2} For every open set $\Omega \subset \Hei$,
 \begin{displaymath}
  \mathrm{BMO}^s(\Omega)=\mathrm{BMO}^{\Hei}(\Omega)
  \end{displaymath}
  with
 \begin{displaymath}
 c_1 \|\cdot\|^{\mathbb{H}^1}_{\ast} \leq \|\cdot\|^s_{\ast} \leq c_2 \|\cdot\|^{\mathbb{H}^1}_{\ast}
 \end{displaymath}
 for  constants $0<c_1 \leq c_2<\infty$ that do not depend on $\Omega$.
 \end{thm}
The equivalence of various BMO spaces is used in the proof of Theorem \ref{c:logBMOdomain} together with distortion estimates that we deduce from the local quasisymmetry property of quasiconformal mappings.

The proof of our main result,  Theorem \ref{t:KoebeHeis}, utilizes the auxiliary results established in Section  \ref{s:JacobianQC}, Theorem \ref{c:logBMOdomain},
quantitative control over the local quasisymmetry data of quasiconformal mappings,
  as well as other observations such as the distance estimate in Proposition~\ref{c:dist_est}. The latter extends \cite[Lemma 5.15]{MR861687} from planar disks to arbitrary domains in $\Hei$.

  We also discuss various applications of Theorem \ref{t:KoebeHeis}, both for the sub-Riemannian and the Kor\'{a}nyi distance:
  \begin{itemize}
  \item Coupled with ball estimates and covering arguments, the Koebe theorem yields quasiconformal versions of results established in \cite{LRj} for quasisymmetries in an abstract setting. In Proposition~\ref{t:diam} we extend a diameter estimate for images of curves under quasiconformal mappings by P.\ Koskela, \cite[Lemma 2.6]{Koskela1994}, to the setting of $\Hei$
and we use
 an $\mathbb{H}^1$ version of the radial stretch mapping
  to
  show the sharpness of this result.
  \item Section \ref{subs:comp} is devoted to proving the comparability relation between the $L^p$-ope{\-}ra{\-}tor norm of the horizontal differential of a quasiconformal mapping and the $L^p$-integral of $a_f$, see Theorem~\ref{t:comparable}. This extends a result by Astala and Koskela \cite{MR1191747} to the Heisenberg setting and it shows how the global integrability properties of the horizontal derivative of a quasiconformal map  on a domain in $\Hei$ depend on the distortion properties encoded by $a_f$.
Amongst others,
our proof requires a specific Whitney decomposition, Lemma~\ref{l:whitney},
which we believe to be of independent interest.
\item  Finally, in Section \ref{sec:metrics}, we apply several of the mentioned results together with
Theorem \ref{t:KoebeHeis} to extend a result of Bonk--Koskela--Rohde~\cite{bkr98} regarding conformal metrics
  and quasiconformal mappings on the unit ball
  to general domains in the Heisenberg group, equipped with either $d_{\Hei}$ or $d_s$. Namely, we prove the following result (see Proposition~\ref{p:a_f_conf_density} for the precise statement):
\end{itemize}
  \begin{thm}\label{t:introThm3}
  If $f:\Omega \to \Omega'$ is a $K$-quasiconformal map between domains $\Omega,\Omega' \subsetneq \Hei$, then $a_f$ satisfies
  \begin{enumerate}
  \item a Harnack inequality,
  \item a growth condition for volume:
  \begin{displaymath}
  \int_{B_{a_f}(x,r)}a_f^4 \,dm \lesssim r^4\quad\text{for all }x\in \Omega,\, r>0,
  \end{displaymath}
  where
 $
  B_{a_f}(x,r):= \{y\in \Omega:\; \inf \int_{\gamma} a_f \,ds<r\}
$ with the infimum taken over all locally rectifiable curves in $\Omega$ that connect $x$ and $y$.
  \end{enumerate}
 The implicit multiplicative constants in (1) and (2) depend only on $K$ and the properties of the metric space $(\Hei,d)$, $d\in \{d_s,d_{\Hei}\}$.
  \end{thm}
% We hope that quasiconformal densities and related metrics will be further investigated in the future.

For quasiconformal maps defined on the entire space $\Hei$, the results mentioned above are either meaningless or already known.
We consider it one of the contributions of the present paper to provide appropriate localizations and to handle the technical difficulties that arise when dealing with maps defined on subdomains $\Omega \subsetneq \Hei$. Our work is inspired by results for quasiconformal maps on disks in the plane, and more generally on domains in $\mathbb{R}^n$.
Several tools available in the Euclidean setting, such as the Teichm\"uller rings used in \cite[Lemma 4]{MR0361067}, an extension results for quasiconformal mappings or the Mori distortion theorem used in \cite{MR861687}, are not available in the Heisenberg group. We show that the local $\eta$-quasisymmetry of quasiconformal maps with a good control over $\eta$ can be used as a substitute for these missing tools.

 \noindent\textbf{Structure of the paper.} In Section \ref{s:prelim} we introduce the most important notions used throughout this paper. We recall some basic information about the Heisenberg group and discuss quasiconformal and quasisymmetric mappings in $\Hei$. Section \ref{s:JacobianQC} is devoted  to the proof of Theorem \ref{c:logBMOdomain}, along the way we also establish Theorem \ref{t:introThm2}. In Section \ref{s:Koebe} we prove our main result, Theorem \ref{t:KoebeHeis}.   We conclude the paper with Section \ref{s:applications}, in which we discuss various applications of Theorem \ref{t:KoebeHeis} that culminate in Theorem \ref{t:introThm3}.

\noindent \textbf{Acknowledgements.} We thank Pekka Koskela for bringing the article \cite{MR1191747} to our attention. Part of the work on the present paper was done while K.F. visited IMPAN in October 2016 and while T.A. and B.W. visited the University of Fribourg in February 2017. We would like to thank the respective hosting institution for creating the scientific atmosphere and support. We are also grateful to the referee for numerous insightful comments that helped to improve the presentation of the manuscript. In particular we acknowledge a remark that prompted us to use local quasisymmetry to establish Proposition \ref{p:ball_dist} and Theorem \ref{t:KoebeHeis}, rather than repeating the modulus arguments  involved in the proof of the local quasisymmetry of quasiconformal maps.

\section{Definitions and preliminaries}\label{s:prelim}

The purpose of this section is to introduce concepts appearing in this paper: Loewner spaces, the Heisenberg group, and quasiconformal mappings (in the Heisenberg group). The definitions given here are standard, and a reader who is familiar with the subject may wish to go directly to Section \ref{s:JacobianQC}.

\subsection{Curves and Loewner spaces}

An important tool in the theory of quasiconformal mappings is the modulus of curve families, discussed in detail for instance in \cite{MR1654771} and in the monographs \cite{MR0454009, martio2008moduli}. Crucial properties of quasiconformal mappings that will be used in this paper, for instance Propositions \ref{p:egg_yolk}  and \ref{p:ball_dist}, are ultimately based on modulus estimates.

By a \emph{curve} in a metric space $(X,d)$ we mean a continuous map $\gamma: I \to X$ of an interval $I\subset \mathbb{R}$. A Borel function $\rho: X \to [0,+\infty]$ can be integrated with respect to arc length along rectifiable curves. For a locally rectifiable curve $\gamma: I \to X$, we set
\begin{displaymath}
\int_{\gamma}\rho\;\mathrm{d}s:= \sup_{\gamma'}\int_{\gamma'}\rho\;\mathrm{d}s,
\end{displaymath}
where the supremum is taken over all rectifiable subcurves $\gamma'$ of $\gamma$.

\begin{definition} \label{defn:curv-mod}
Let $(X,d)$ be a metric space and let $\mu$ be a Borel measure on $X$. The \emph{admissible densities} of a family $\Gamma$ of curves in $X$ are defined as
\begin{displaymath}
\mathrm{adm}(\Gamma):= \left\{\rho:X\to [0,+\infty]\text{ Borel\,: }\int_{\gamma}\rho\;\mathrm{d}s\geq 1\text{ for all }\gamma\in\Gamma\text{ locally rectifiable}\right\}.
\end{displaymath}
The \emph{$p$-modulus} of $\Gamma$ for $p\geq 1$, of $\Gamma$ is given by
\begin{displaymath}
\mathrm{mod}_p(\Gamma):= \inf \left\{\int_X \rho^p\;\mathrm{d}\mu:\; \rho\in\mathrm{adm}(\Gamma)\right\}.
\end{displaymath}
\end{definition}

%We will often use the modulus of a family that consists of all curves in $X$ which connect two sets $E$ and $F$. We denote such a family by $\Gamma(E,F,X)$.

The family of all curves in $X$ connecting two sets $E$ and $F$ is denoted by $\Gamma(E,F,X)$.

%DELETED:
%The modulus of curve families provides a way to measure how well a space is connected by short rectifiable curves. We recall the following definition from \cite{MR1654771}.

\begin{definition} Let $(X,d)$ be a rectifiably connected metric space of Hausdorff dimension $Q\geq 1$ and assume that $X$ is endowed with a
locally finite Borel regular measure $\mu$  with dense support. Then $X$ is said to be a \emph{($Q$-)Loewner space} if for all $t\in (0,\infty)$ one has
\begin{equation}\label{eq:LoewnerFunction}
\psi(t):=\inf \left\{ \mathrm{mod}_Q \Gamma(E,F,X):\; \triangle(E,F):= \frac{\mathrm{dist}(E,F)}{\min \{\mathrm{diam} E, \mathrm{diam} F\}}\leq t\right\}>0,
\end{equation}
where the infimum is taken over  disjoint nondegenerate continua
%ADDED:
(compact and connected sets)
 $E$ and $F$ in $X$. We call the function $\psi$ the \emph{Loewner function} of $(X,d,\mu)$.
\end{definition}

\subsection{The Heisenberg group}\label{ss:Heis}

The first Heisenberg group $\mathbb{H}^1$ is a noncommutative nilpotent Lie group homeomorphic to $\mathbb{R}^3$. It can be endowed with a left-invariant distance $d$ such that $(\mathbb{H}^1,d)$ does not biLipschitzly embed into any Euclidean space, yet exhibits a rich and interesting geometry.
%DELETED:
%Comparable metrics on $\mathbb{H}^1$ can be obtained from different areas of mathematics such as control theory or complex hyperbolic geometry, where the Heisenberg group appears in a priori unrelated contexts.
For an introduction to the subject,  we refer the interested reader to the monograph \cite{capogna2007introduction}.

Our model for $\mathbb{H}^1$ is the group $(\R^3, *)$ where the group law is given by
\begin{displaymath}
(x,y,t) * (x',y',t')=(x+x',y+y',t+t'-2xy' + 2x'y).
\end{displaymath}
Using this group law, one defines a frame of left-invariant vector fields which agree with the standard basis at the origin:
\begin{displaymath}
X:= \partial_x + 2y\partial_t,\quad Y: = \partial_y-2x\partial_t,\quad T:=\partial_t.
\end{displaymath}
The vector fields $X$ and $Y$, which are called \emph{horizontal}, have a non-vanishing commutator $[X,Y]=-4T$. This ensures that any two points $p$ and $q$ in $\mathbb{H}^1$ can be connected by an absolutely continuous curve $\gamma:[0,1]\to\mathbb{H}^1$ with the property that
\begin{displaymath}
\dot{\gamma}(s)\in H_{\gamma(s)},\quad \text{a.e. }s\in [0,1],\text{ where }H_p:= \mathrm{span}\{X_p,Y_p\}.
\end{displaymath}
Such a $\gamma$ is called a \emph{horizontal curve}. The \emph{sub-Riemannian distance} $d_s$ is defined by
\begin{displaymath}
d_s(p,q)=\inf_{\gamma} \int_0^1 \sqrt{\dot{\gamma}_1(s)^2 + \dot{\gamma}_2(s)^2}\;\mathrm{d}s,
\end{displaymath}
where the infimum is taken over all horizontal curves $\gamma=(\gamma_1,\gamma_2,\gamma_3):[0,1]\to \mathbb{H}^1$ that connect $p$ and $q$. It is well known that $d_s$ defines a left-invariant metric on $\mathbb{H}^1$ which is homogeneous under the \emph{Heisenberg dilations}  $(\delta_{\lambda})_{\lambda>0}$, given by
\begin{displaymath}
\delta_{\lambda}:\mathbb{H}^1 \to \mathbb{H}^1,\quad \delta_{\lambda}(x,y,t)=(\lambda x,\lambda y,\lambda^2 t)\,\hbox{ for } (x,y,z)\in \Hei.
\end{displaymath}
Any two homogeneous left-invariant metrics on $\mathbb{H}^1$ are bi-Lipschitz equivalent, and it is often more convenient to work with a left-invariant metric which is given by an explicitly computable formula, rather than to use $d_s$. An example of such a metric is the \emph{Kor\'{a}nyi distance}, defined by
\begin{displaymath}
d_{\mathbb{H}^1}(p,q) := \|q^{-1}p\|_{\mathbb{H}^1},\quad \text{where }\,\,\|(x,y,t)\|_{\mathbb{H}^1}= \sqrt[4]{(x^2+y^2)^2 + t^2}.
\end{displaymath}
For all $p,q\in \Hei$ it holds that
\begin{equation}\label{eq:comp_metric}
 \frac{1}{\sqrt{\pi}} d_s(p, q) \leq d_{\Hei}(p,q)  \leq d_s(p, q),
\end{equation}
see~\cite{bel}, and the length distance associated to $d_{\mathbb{H}^1}$ is exactly $d_s$.

In addition to the metric structure, we endow the Heisenberg group with a bi-invariant Haar measure $m$ which is given by the Lebesgue measure on $\mathbb{R}^3$. We recall that this measure $m$ is Ahlfors $4$-regular.
%DELETED:
%and satisfies the annular decay property.
It agrees, up to a positive and finite multiplicative factor, with the $4$-dimensional Hausdorff measure with respect to a left-invariant homogeneous distance on $\mathbb{H}^1$.
Unless otherwise stated, ``measurable'' and ``integrable'' will in the following always mean ``$m$ measurable'' and ``$m$ integrable''. We denote $m(A)=:|A|$ for $A\subseteq \mathbb{H}^1$, and we write $\int f \;\mathrm{d}m= \int f(x) \;\mathrm{d}x$. Equipped with $m$ and any homogeneous left-invariant distance, the Heisenberg group becomes a $4$-Loewner space, see for instance~\cite[\S 9.25]{heinonen2012lectures}.
%\cite{capogna2007introduction}.

\begin{quote}\textbf{Convention.} Whenever we discuss quantitative dependencies of parameters on certain constants, we will omit information that such constants may also depend on the data of the metric measure space $(\mathbb{H}^1,d_{\mathbb{H}^1},m)$ or $(\mathbb{H}^1,d_s,m)$. For instance, if we say that ``a constant $C$ depends only on the distortion $K$ of the mapping'', the constant $C$ may depend also on the Loewner function, the quasiconvexity and doubling constants, etc associated to $d\in \{d_s,d_{\Hei}\}$.\\
As remarked in the introduction, Theorem \ref{t:KoebeHeis} for $d_{\mathbb{H}^1}$ is equivalent to the analogous statement with respect to $d_s$. The same holds true for the applications (Proposition \ref{t:diam}, Theorem \ref{t:comparable}, Proposition \ref{p:a_f_conf_density}). For auxiliary results needed in these proofs, we will always specify whether they hold with respect to $d_s$, $d_{\mathbb{H}^1}$, or both.\end{quote}

\subsection{Quasiconformal and quasisymmetric mappings}\label{subs-2-3}

In this section we collect the relevant facts about quasiconformal mappings in the Heisenberg group. Quasiconformal maps can be defined primarily by three definitions, the metric, analytic and geometric definition, all of which are mutually and quantitatively equivalent on domains in
$\mathbb{H}^1$, even though the distortion factor need not be the same for each definition.  The equivalence of these definitions is a central part of the general theory of quasiconformal maps and we refer the reader to \cite{HeinonenNotes,  MR1654771, MR1869604,tys2001} for details at a general level and \cite{MR1317384} for the specific case of the Heisenberg group. An important feature to note is that the class of metrically defined quasiconformal maps is the same for any pair of bi-Lipschitz equivalent metrics with a quantitative control on the distortion. As remarked in the introduction, if the two metrics are the sub-Riemannian distance $d_s$ and the Kor\'{a}nyi metric $d_{\Hei}$, then one  gets even the same  class of $K$-quasiconformal maps.
Thus, in our context it often does not matter if we use the sub-Riemannian metric  or the Kor\'{a}nyi metric and so we leave the metric unspecified in the respective statements.

In order to state the metric definition of quasiconformal mappings we introduce the following notation. Let  $\Omega \subseteq \Hei$ be an open set and let further $f:\Omega \to f(\Omega)\subseteq \mathbb{H}^1$ be a homeomorphism. For all $p \in \Omega$ and all $r>0$ we define
	\begin{align*}
	L_f (p, r) &:= \sup\{d (f(p), f(q)) : q\in \Omega,\,d(p, q) \leq r\},\\
	l_f (p, r) &:= \inf\{d(f(p), f(q)) : q\in\Omega,\,d(p, q) \geq r\}, \hbox{ and }\\
	H_f (p) &:= \limsup_{r \to 0} \frac{L_f (p, r)}{l_f(p,r)}.
	\end{align*}
\begin{definition}[Metric definition]\label{QCmetdef}
	We say that a homeomorphism $f:\Omega \to f(\Omega)\subseteq \mathbb{H}^1$ of an open set $\Omega \subseteq \Hei$ is \emph{quasiconformal}, if $H_f$ is bounded  on  $\Omega$.
\end{definition}
While metric quasiconformality is an infinitesimal property, \emph{quasisymmetry} is a global and generally stronger condition.

\begin{definition}[Quasisymmetric definition]\label{QCqsymdef}
 If $\Omega$ is an open set in $\Hei$ and $\eta:[0,\infty) \to [0,\infty)$ is a homeomorphism, then we say that a homeomorphism $f : \Omega \to  f(\Omega)\subseteq \Hei$ is \emph{$\eta$-quasisymmetric} if
\begin{align}
 \frac{d(f(p_1),f(p_2))}{d(f(p_1),f(p_3))} \leq \eta(t)  \label{etaqsymcond}
\end{align}
 for all $t>0$ and all triples of distinct points $p_1,p_2,p_3 \in \Omega$ satisfying $d(p_1,p_2) \leq t d(p_1,p_3)$. A map $f$ is said to be \emph{quasisymmetric} if it is $\eta$-quasisymmetric for some $\eta$.\\
 \medskip
 We say that $f$ is \emph{(weakly) $H$-quasisymmetric} if there exists a constant $H\geq 1$ such that
 \begin{align}
 \frac{d(f(p_1),f(p_2))}{d(f(p_1),f(p_3))} \leq H  \label{etaqsymcond}
\end{align}
 for all  triples of distinct points $p_1,p_2,p_3 \in \Omega$ satisfying $d(p_1,p_2) \leq d(p_1,p_3)$.
\end{definition}

A quasiconformal map defined on all of $\Hei$ is $\eta$-quasisymmetric for some $\eta$ that depends on the quasiconformal distortion \cite{zbMATH00746069}. An analogous statement is not true in general for mappings defined on a subdomain of $\Hei$, but the metric definition still implies a local quasisymmetry condition in the sense of Theorem \ref{qc-implies-qs} and Proposition \ref{p:egg_yolk} below. This goes back to \cite[Proposition 22]{MR1317384}.
Theorem \ref{qc-implies-qs} was proven (for  Carnot groups of dimension at least $2$)
 by Heinonen and Koskela \cite[Theorem 1.3]{zbMATH00746069} for globally defined maps, but
Heinonen remarked  in \cite[p.25]{HeinonenNotes} that the argument can be adapted to mappings between open subsets. The details for the proof showing that $\eta$ can be chosen independently of the domain $\Omega$ are given in \cite{MR3029176}; see also the comment below.

\begin{thm}[Heinonen, Koskela] \label{qc-implies-qs} For every $K\geq 1$, there exists a homeomorphism $\eta:[0,\infty) \to [0,\infty)$ such that the following holds.
 If $f$ is a $K$-quasiconformal mapping of an open set $\Omega \subseteq (\Hei,d_s)$ according to Definition~\ref{QCmetdef}, then for all triples $p,q_1,q_2 \in \Omega$ with $q_1,q_2 \in B(p, \frac{1}{2}d_s(p, \partial \Omega) )$, the mapping $f$ satisfies:
\begin{align*}
 \frac{d_s(f(p),f(q_1))}{d_s(f(p),f(q_2))} \leq \eta\left(\frac{d_s(p, q_1)}{d_s(p, q_2)}\right).
\end{align*}
\end{thm}
By the triangle inequality, the ``$p$-centered'' quasisymmetry property in Theorem \ref{qc-implies-qs} implies quasisymmetry of $f$ on the sub-Riemannian ball $B(p, \frac{1}{5}d_s(p, \partial \Omega) )$. Following the terminology in \cite[p.93]{heinonen2012lectures}, we call this fact an ``egg yolk principle''. {E.\ Soultanis and M.\ Williams \cite[Lemma 5.2]{MR3029176} provided a proof for this principle in great generality, and with a quantitative control both on the size of the ``egg yolk'' and the $\eta$-function in the definition of local quasisymmetry. See Proposition \ref{p:egg_yolk_cor} below for a related corollary. Naturally, one can also establish the local quasisymmetry property for the metric $d_{\Hei}$ instead of $d_s$ (with a possibly different homeomorphism $\eta$). This is because the comparability of the metrics $d_s$ and $d_{\Hei}$ as
stated in \eqref{eq:comp_metric} implies that
\begin{equation}\label{eq:Whitney_ball_inclusion}
B_{\Hei}\left(p,\frac{d_{\Hei}(p,\partial \Omega)}{5 \sqrt{\pi}}\right) \subseteq B_s\left(p,\frac{d_s(p,\partial \Omega)}{5}\right)
\end{equation}
for all domains $\Omega \subsetneq \Hei$ and all $p\in \Omega$. We record these observations as a  proposition, for which we do not claim any novelty.

\begin{proposition}\label{p:egg_yolk}
Let $K\geq 1$ and $d\in \{d_s,d_{\Hei}\}$. Then there exists a homeomorphism $\eta: [0,+\infty) \to [0,+\infty)$ such that  every $K$-quasiconformal map $f:\Omega \to \Omega'$ between domains in $\Hei$ is \emph{locally $\eta$-quasisymmetric} in the following sense. For every $p\in \Omega$, the map $f$ restricted to the $d$-ball $B(p,d(p,\partial \Omega)/c_d)$ is $\eta$-quasisymmetric, where $c_{d}=5$ for $d=d_s$, and $c_{d} = 5\sqrt{\pi}$ for $d=d_{\Hei}$.
\end{proposition}}
Quasiconformal maps also exhibit useful analytic properties. It was shown by
G.\
Mostow that a quasiconformal map on a domain in $\mathbb{H}^1$ is \emph{absolutely continuous on lines (ACL)}, see the discussion in \cite{MR1317384}. This property is defined analogously as the ACL property for mappings on open subsets of $\mathbb{R}^n$, but in terms of the fibrations given by the left invariant horizontal vector fields $X$ and $Y$ instead of lines parallel to the coordinate axes.
In \cite{Pansu1}, P.\ Pansu showed that local quasisymmetry for a map $f$ on an open subset of $\Hei$ implies further analytic features similar to those of quasiconformal mappings on domains in $\mathbb{R}^n$:
 for a quasiconformal map $f$ on $\Omega$ the \emph{horizontal derivatives}  $X f(p)$ and $Y f(p)$ exist for almost every $p \in \Omega$, and $f$ is \emph{Pansu differentiable}
almost everywhere in $\Omega$.

We define the \emph{Jacobian} $J_f(p)$ of a quasiconformal map $f$ at $p\in \Om$ as the volume derivative
$$
J_f(p) =\limsup_{r \to 0} \frac{|f(B(p,r))|}{|B(p,r)|} \quad\hbox{for } p \in \Omega.
$$
According to Lebesgue's differentiation theorem, the $\limsup$ can be replaced by $\lim$ in almost every point $p\in \Omega$.
If a quasiconformal map $f=(f_1,f_2,f_3)$ is $P$-differentiable at $p$, then
$J_f(p) = \det D_H f(p)^2$, where
$$
D_Hf(p) =\left( \begin{array}{cc}
Xf_1  & Yf_1 \\
Xf_2  & Yf_2   \end{array} \right).
$$
The analytic definition of quasiconformal mappings can now be stated as follows.

\begin{definition}[Analytic definition]\label{QCandef}
	If $\Omega$ is an open set in $\Hei$, we say that a homeomorphism $f : \Omega \to  f(\Omega)\subset \Hei$ is \emph{$K$-quasiconformal} if it is $\mathrm{ACL}$, Pansu differentiable almost everywhere,  and satisfies the following distortion condition: there exists $1 \leq K < \infty$ such that
\begin{equation}\label{eq:analyt_dist_ineq}
\|D_Hf (p)\|^4  \leq K J_f(p) \,\,\hbox{ for almost every }p\in\Om,
\end{equation}
where
$$
\|D_H f(p)\| = \max \{ |D_H f(p) \xi| \, :\, \xi \in H_p, \, \,|\xi|=1\}
$$
and $|\cdot|$ is obtained from the inner product which makes $\{X,Y\}$ orthonormal. A map $f$ is said to be \emph{quasiconformal} if it is $K$-quasiconformal for some $1\leq K<\infty$.
\end{definition}
Quasiconformal mappings are also absolutely continuous in measure (Proposition 3 in \cite{MR1317384}) with $J_f>0$ almost everywhere on $\Omega$.
This allows to show that the metric and analytic definition for quasiconformal mappings between domains in $\Hei$ are quantitatively equivalent to a third condition,  the \emph{geometric definition}, based on the $4$-modulus of curve families. We have decided to define ``$K$-quasiconformal'' through the analytic distortion inequality \eqref{eq:analyt_dist_ineq}. This is a matter of taste, but is convenient due to the following implications for a homeomorphism $f$ on $\Omega \subset \Hei$:
\begin{equation}\label{eq:equivQC}
\text{$K$-quasiconformal}\quad\Leftrightarrow\quad\text{metrically quasiconformal with }\mathrm{ess sup}_{p\in \Omega} H_f(p)\leq \sqrt{K}.
\end{equation}
Finally, we state a theorem due to Kor\'anyi and Reimann \cite[Proposition 20]{MR1317384}.
\begin{thm}[Kor\'{a}nyi, Reimann]
The inverse of a $K$-quasiconformal mapping between domains in $\Hei$ is $K$-quasiconformal.
\end{thm}
In \cite{MR1317384}, the metric definition was used to define ``$K$-quasiconformal'', but it can be seen from the proof, or by applying \eqref{eq:equivQC}, that the statement holds equivalently if the distortion is defined via the analytic definition as done in the present paper.

\section{$\mathrm{BMO}$ spaces and Jacobians of quasiconformal mappings}\label{s:JacobianQC}
It is well known that the Jacobian $J_f$ of a quasiconformal map $f:\mathbb{R}^n \to \mathbb{R}^n$ is an $A_{\infty}$-weight and hence $\log J_f$ is of bounded mean oscillation (BMO).
The situation is more subtle if one considers quasiconformal maps on a subdomain $\Omega \subset \mathbb{R}^n$. As shown in \cite{MR1169029}, it is not true in general that $J_f$ is an $A_{\infty}$-weight on $\Omega$, but even so $\log J_f$ lies in the  (appropriately defined) space $\mathrm{BMO}(\Omega)$. The goal of this section is to extend the latter statement from $\mathbb{R}^n$ to $\mathbb{H}^1$ by proving Theorem \ref{c:logBMOdomain}. We start with the relevant definitions.
For further reading, a classical reference for BMO spaces on homogeneous groups is \cite{MR657581}.
\subsection{BMO spaces on domains in $\Hei$}

\begin{definition}\label{d:BMO_loc}
Let $\Omega$ be an open subset of $\mathbb{H}^1$. We say that a function $u\in L^1_{loc}(\Omega)$ belongs to $\mathrm{BMO}^s(\Omega)$ if there is a constant $C$ such that
\begin{displaymath}
\Barint_B |u-u_B|\;\mathrm{d}m \leq C,\quad\text{for every $d_s$-ball }B\subseteq \Omega.
\end{displaymath}
The space $\mathrm{BMO}^{\Hei}(\Omega)$ is defined analogously with ``$d_s$'' replaced by ``$d_{\Hei}$''.
\end{definition}

\begin{definition}\label{d:BMO-norm}
For a domain $\Omega \subseteq \mathbb{H}^1$ and a function $u\in\mathrm{BMO}^s(\Omega)$, we define the
% ADDED (semi)
\emph{$\mathrm{BMO}^s(\Omega)$-(semi)norm} as
\begin{displaymath}
\|u\|_{\ast}^s:= \sup_{B} \Barint_B |u-u_B|\;\mathrm{d}m,
\end{displaymath}
where the supremum is taken over all $d_s$-balls $B\subset \Omega$.\\
The seminorm $\|\cdot\|_{\ast}^{\Hei}$ is defined analogously with ``$d_s$'' replaced by ``$d_{\Hei}$''.
\end{definition}

If the choice of metric $d\in \{d_s,d_{\Hei}\}$ is clear from the context (or irrelevant) we sometimes omit the superscript simply write $\mathrm{BMO}(\Omega)$ and $\|\cdot\|_{\ast}$.

We will prove Theorem \ref{c:logBMOdomain}, which states that $\log J_f \in \mathrm{BMO}(\Omega)$  for every quasiconformal map $f$ defined on a domain $\Omega \subset \Hei$, and moreover, $\|\log J_f\|_{\ast}$ can be bounded in terms of the distortion constant of $f$.
The outline of the proof  follows its Euclidean predecessors in \cite{MR0361067,MR0511997}. The main technical difficulty stems from the fact that we consider mappings which might be defined only on a subdomain of $\mathbb{H}^1$, and we work both with the sub-Riemannian distance and the Kor\'{a}nyi metric. For mappings of the entire Heisenberg group, it is well known that $J_f$ is a Muckenhoupt $A_p$-weight for some $1\leq p<\infty$, and hence an $A_{\infty}$-weight. This is a consequence of a `reverse H\"older inequality' due to Kor\'{a}nyi and Reimann (Theorem \ref{t:GehringHeisenberg} below), see for instance the overview in Section 3 of \cite{Aus}.
In Section \ref{s:BMO}, we discuss various local BMO spaces in the setting of the Heisenberg group $\Hei$, based on a characterization of $\mathrm{BMO}$ spaces in doubling length metric spaces due to S.\ Buckley.
Sections \ref{s:revHol} and \ref{s:weights} contain properties of  the Jacobian $J_f$ of a quasiconformal mapping, which are used to deduce in Section~\ref{s:logJacobian-proof} that $\log J_f$ belongs to a certain local $\mathrm{BMO}$ space. This in turn gives us the proof of Theorem~\ref{c:logBMOdomain}.

\subsection{Local BMO spaces on domains in $\Hei$}\label{s:BMO}

 The goal of this section is to study the local $\mathrm{BMO}$ spaces defined with respect to distance functions $d_{\mathbb{H}^1}$ and $d_s$. Among the results we show that all the respective spaces agree.
\begin{definition}\label{d:BMO_loc}
Let $\Omega$ be an open subset of $\mathbb{H}^1$.
We say that $u\in L^1_{loc}(\Omega)$ belongs to the \emph{local $n$-$\mathrm{BMO}$ space} $\mathrm{BMO}^s_{n,loc}(\Omega)$ for $n \geq 1$ if there is a constant $C$ such that
\begin{equation}\label{d:bmoloc}
\Barint_B |u-u_B|\;\mathrm{d}m \leq C,\quad\text{for every $d_s$-ball $B$ with }nB \subseteq \Omega.
\end{equation}
We say that $u\in L^1_{loc}(\Omega)$ belongs to $\mathrm{BMO}^s_{loc}(\Omega)$ if there is $n>1$ such that $u\in \mathrm{BMO}^s_{n,loc}(\Omega)$.\footnote{In our application, the constant $n$ will be determined by the proof. In the standard definition of  $\mathrm{BMO}_{loc}(\Omega)$ one would take $n=2$, as for instance in \cite{MR1749313}.}

The spaces $\mathrm{BMO}_{n,loc}^{\Hei}(\Omega)$ and $\mathrm{BMO}_{loc}^{\Hei}(\Omega)$ are defined analogously with ``$d_s$'' replaced by ``$d_{\Hei}$''.
\end{definition}

As before, we may write $\mathrm{BMO}_{n,loc}(\Omega)$ and $\mathrm{BMO}_{loc}$ without specifying the metric $d\in \{d_s,d_{\Hei}\}$.

\begin{definition}\label{d:BMO-norm-loc}
For a domain $\Omega \subseteq \mathbb{H}^1$ and a function $u\in\mathrm{BMO}_{n,loc}(\Omega)$, we define the
\emph{{$\mathrm{BMO}_{n,loc}(\Omega)$}-(semi)norm} as
\begin{displaymath}
\|u\|_{\mathrm{BMO}_{n,loc}(\Omega)}:= \sup_{B} \Barint_B |u-u_B|\;\mathrm{d}m,
\end{displaymath}
where the supremum is taken over all $d$-balls $B$ which satisfy $nB\subset \Omega$.
\end{definition}

The following lemma addresses some of the claims made in the introduction.

\begin{lemma}\label{c:bmo-sub-Hei}
For all open sets $\Omega \subseteq \mathbb{H}^1$, for all $n\geq 1$ and for all $u\in L^1_{loc}(\Omega)$
\[
\|u\|_{{\rm BMO}_{\sqrt{\pi}n,loc}^s(\Omega)}\leq 2\pi^2 \|u\|_{{\rm BMO}_{n,loc}^{\mathbb{H}^1}(\Omega)}
\quad\hbox{ and }\quad
\|u\|_{{\rm BMO}_{\sqrt{\pi}n,loc}^{\mathbb{H}^1}(\Omega)}\leq 2\pi^2 \|u\|_{{\rm BMO}_{n,loc}^{s}(\Omega)}.
\]
In particular, one has
 \begin{equation}
  {\rm BMO}_{n, loc}^{\Hei}(\Om)\subset {\rm BMO}_{\sqrt{\pi}n, loc}^{s}(\Om)\quad\text{and}\quad {\rm BMO}_{n', loc}^{s}(\Om) \subset {\rm BMO}_{\sqrt{\pi}n', loc}^{\Hei}(\Om),
  \label{in:loc-loc}
 \end{equation}
 for all $n,n'\geq 1$.
\end{lemma}

\begin{proof}
We prove the first inclusion in \eqref{in:loc-loc} and the estimate for the corresponding seminorms. The proof of the remaining claims follows the same lines. The argument uses the precise relation between $d_s$ and $d_{\mathbb{H}^1}$ stated in \eqref{eq:comp_metric}. Let us denote by $B=B_s(x,r)\subset \Om$ a ball defined with respect to the sub-Riemannian distance such that $\sqrt{\pi}nB\subset \Om$, for a given $n\geq 1$. By the above relation between distances, there exists a ball $B'=B_{\Hei}(x,r)$, defined with respect to the Kor\'{a}nyi distance, satisfying $B'\subset \Om$ and such that $B\subset B'$ and $nB'\subset \Om$. We verify by direct computations that for any function $u\in {\rm BMO}_{n, loc}^{\Hei}(\Om)$ it holds that
%DELETED:
\begin{align*}
 |u_{B'}-u_{B}|
 %&
 %=\left|u_{B'}-\frac{1}{|B|}\int_{B}u\;\mathrm{d}m\right|
 =\frac{1}{|B|}\left|\int_{B}(u_{B'}-u)\;\mathrm{d}m\right|
 %\\&\leq
 %\frac{1}{|B|}\int_{B}\left|u_{B'}-u\right|\;\mathrm{d}m
 \leq \frac{c}{|B'|}\int_{B'}\left|u_{B'}-u\right|\;\mathrm{d}m,
\end{align*}
where $c=\pi^2$. From this estimate we infer the following inequality:
\begin{align}\label{e:comp_BMO_diff_metric}
\notag \frac{1}{|B|}\int_{B}\left|u-u_{B}\right|\;\mathrm{d}m& \leq  \frac{1}{|B|}\int_{B}\left|u-u_{B'}\right|+|u_{B}-u_{B'}| \;\mathrm{d}m\\
 & \leq \frac{c}{|B'|}\int_{B'}\left|u-u_{B'}\right|\;\mathrm{d}m+\frac{c}{|B'|}\int_{B'}\left|u_{B'}-u\right| \;\mathrm{d}m\\
\notag & \leq \frac{2c}{|B'|}\int_{B'}\left|u-u_{B'}\right|\;\mathrm{d}m.
\end{align}
Applying this reasoning to all sub-Riemannian balls with $\sqrt{\pi}nB\subseteq \Omega$, it follows that $u\in {\rm BMO}_{\sqrt{\pi}n, loc}^{s}(\Om)$ provided that $u\in {\rm BMO}_{n, loc}^{\mathbb{H}^1}(\Om)$, with the desired bound for the $\mathrm{BMO}$-norm.

The second inclusion in \eqref{in:loc-loc} with the corresponding estimate for the BMO-norm follows the same reasoning with $B:= B_{\mathbb{H}^1}(x,r)$, $B':=B_s(x,\sqrt{\pi}r)$ and $n$ replaced by $n'$.
\end{proof}

\begin{remark}
Lemma \ref{c:bmo-sub-Hei} is a special case of a more general result holding in doubling metric measure spaces for two given bi-Lipschitz equivalent metrics, as can be seen from the proof. Our focus here lies on the two metric $d_s$ and $d_{\Hei}$ and the precise relations that can be derived from \eqref{eq:comp_metric}. For more information about BMO spaces in doubling length metric spaces and their applications, we refer the interested reader to \cite{MR1749313} and the references therein.
\end{remark}

We will later see that $\mathrm{BMO}^s(\Omega) = \mathrm{BMO}^s_{n, loc}(\Omega)$ for all $n>1$.
Lemma~\ref{c:bmo-sub-Hei} then implies ${\rm BMO}^{\mathbb{H}^1}(\Omega)\subseteq {\rm BMO}^s(\Omega)$, but the reverse inclusion  does not follow directly from the lemma
since this would require to consider arbitrary balls contained in $\Omega$. To discuss this, we follow Staples~\cite[Definition 2.1]{staples}, and introduce the following definition.

\begin{definition}\label{def-Lp-ave}
Consider the metric measure space $(\mathbb{H}^1,d_s,m)$. A domain $D$ in this space is said to be an \emph{$L^1$-averaging domain} if $m(D)<\infty$  and for some $n>1$ we have
\begin{displaymath}
\frac{1}{|D|}\int_D |u(x)-u_D|\;\mathrm{d}m\leq C_{ave} \left(\sup_{nB_s \subset D}\frac{1}{|B_s|}\int_{B_s}|u(x)-u_{B_s}|\;\mathrm{d}m\right)
\end{displaymath}
for a constant $0<C_{ave}<\infty$ which does not depend on $u$. The supremum in this definition is taken over all sub-Riemannian balls $B_s$ for which the enlarged ball $nB_s$ is contained in $D$.
\end{definition}

Staples defines more generally $L^p$-averaging domains for $p\geq 1$ in the  setting of homogeneous spaces in the sense of Coifman and Weiss, see Section 2 in \cite{staples} for details.

We wish to show that all Kor\'{a}nyi balls $D$ in $\mathbb{H}^1$ are $L^1$-averaging domains with uniform constants $C_{ave}$ and $n$. To this end we will show that the unit ball is an $L^1$-averaging domain, and then conclude by left-invariance and homogeneity. The $L^1$-averaging property of $D$ is a consequence of the geometric condition we now discuss.

Following Definition 2.1 in~\cite{bukolu} we say that a domain $D$ in $(\mathbb{H}^1,d_s,m)$ satisfies the \emph{Boman chain condition} if there exist constants $M>0$, $\lambda>1$, $C_2>C_1>1$, and, a family $\mathcal{F}$ of $d_s$-balls such that \\
(1) $D=\bigcup_{B\in \mathcal{F}} C_1B$\, (the domain $D$ is the union of enlarged balls in $\mathcal{F}$).\\
(2) $\sum_{B\in \mathcal{F}} \chi_{_{C_2B}}(x) \leq M \chi_{_{D}}(x)$ for all $x\in \Hei$\,(a point in $D$ is covered by at most $M$ enlarged balls in $\mathcal{F}$).\\
(3) there exists a so-called ``central ball'' $B_{*}\in \mathcal{F}$ such that for each ball $B\in \mathcal{F}$ there is a positive integer $k=k(B)$ and a chain of balls $B_0, \ldots, B_{k}$ such that $B_0=B, B_k=B_{*}$ and the following properties hold:
\begin{itemize}
\item[(3.1)]  for every $j=0,\ldots, k-1$ there exists a ball $D_j$ satisfying
\[
 D_j\subset C_1B_j\cap C_1B_{j+1}\quad\hbox{and}\quad m(B_j)\approx m(D_j) \approx m(B_{j+1}).
\]
\item[(3.2)] $B\subset \lambda B_j$ for all $j=0,\ldots, k(B)$.
\end{itemize}

Theorem 3.1 in \cite{bukolu}, stated for more general homogeneous metric spaces, in particular shows that John domains in  $(\mathbb{H}^1,d_s,m)$  satisfy the Boman chain condition.

\begin{proposition}\label{p:UnitBallAveraging}
The Kor\'{a}nyi unit ball $D=B_{\mathbb{H}^1}(0,1)$ is an $L^1$-averaging domain in $(\Hei,d_s,m)$.
\end{proposition}

\begin{proof}
First, $D$ is a John domain in $(\Hei,d_s)$. To prove this, it suffices to show that there exists a point $p_0$ (the ``John center'') and a constant $C\geq 1$ such that every $p\in D$ can be connected to $p_0$ by a rectifiable curve $\gamma$ with the property that
\begin{displaymath}
d_s(\gamma(t),\partial D) \geq C^{-1}\min \{d_s(p_0,\gamma(t)),d_s(\gamma(t),p)\}
\end{displaymath}
for all $\gamma(t)$. Clearly we can take $p_0=0$. For $p\in D$, we let $\gamma$ be a (sub-Riemannian) geodesic with $\gamma(0)=0$ and $\gamma(1)=p$. Then, it suffices to observe that for all $q\in B_s(\gamma(t),1-t)$, one has
\begin{displaymath}
d_{\Hei}(q,0)\leq d_{\Hei}(q,\gamma(t))+ d_{\Hei}(\gamma(t),0) \leq d_s(q,\gamma(t))+d_s(\gamma(t),0) \leq 1-t + t = 1.
\end{displaymath}
 By the discussion preceding the statement of Proposition \ref{p:UnitBallAveraging}, the John domain $D$ satisfies the Boman chain condition.
Theorem 4.2 in~\cite{staples} implies that every domain which satisfies the Boman chain condition is an $L^r$-averaging domain for every $1\leq r<\infty$, and thus, in particular for $r=1$. Hence $D$ is an $L^1$-averaging domain as claimed.
%We note that the constant $C_{ave}$ of the averaging domain depends on the John constant and the John center.
\end{proof}

\begin{cor}\label{lem-kor-ave}
Every Kor\'{a}nyi ball $B_{\mathbb{H}^1}(p,r)$ is an $L^1$-averaging domain with the same constants $C_{ave}$ and $n$.
\end{cor}

\begin{proof}
Denote $D_0:=B_{\mathbb{H}^1}(0,1)$. By Proposition \ref{p:UnitBallAveraging}, there exist  $0<C_0<\infty$ and $n_0>1$  such that
\begin{equation}\label{eq:UnitAveraging}
\frac{1}{|D_0|}\int_{D_0} |u(x)-u_{D_0}|\;\mathrm{d}m\leq C_{0} \left(\sup_{n_0B_s \subset D_0}\frac{1}{|B_s|}\int_{B_s}|u(x)-u_{B_s}|\;\mathrm{d}m\right).
\end{equation}
Consider now an arbitrary Kor\'{a}nyi ball $D:= B_{\mathbb{H}^1}(p,r)$. Recall that left translations, denoted $\tau_p$, have Jacobian determinant equal to $1$, and dilations by $r$, denoted $\delta_r$, are diffeomorphisms with Jacobian $r^4$. This yields by the transformation formula that
\begin{displaymath}
u_D:= \frac{1}{|D|}\int_D u \;\mathrm{d}m = \frac{1}{r^4}\frac{1}{|D_0|} \int_{D_0} (u\circ \tau_p \circ \delta_r) r^4\;\mathrm{d}m= (u\circ \tau_p\circ \delta_r)_{D_0}.
\end{displaymath}
Hence, by \eqref{eq:UnitAveraging},
\begin{align*}
\frac{1}{|D|}\int_{D} |u(y)-u_{D}|\;\mathrm{d}y
%&= \frac{1}{r^4|D_0|}\int_{D_0} |(u \circ \tau_p \circ \delta_r)(x)-u_{D}|r^4\;\mathrm{d}x\\
%&= \frac{1}{|D_0|}\int_{D_0} |(u \circ \tau_p \circ \delta_r)(x)-(u\circ \tau_p\circ \delta_r)_{D_0}|\;\mathrm{d}x\\
&= \frac{1}{|D_0|} \int_{D_0}|v(x)-v_{D_0}|\;\mathrm{d}x\\
&\leq C_{0} \left(\sup_{n_0B_s \subset D_0}\frac{1}{|B_s|}\int_{B_s}|v(x)-v_{B_s}|\;\mathrm{d}m\right)\\
%&= C_{0} \left(\sup_{n_0B_s \subset D_0}\frac{1}{|\tau_p \delta_r B_s|}\int_{\tau_p \delta_r B_s}|u(x)-u_{\tau_p \delta_r B_s}|\;\mathrm{d}m\right)\\
&=  C_{0} \left(\sup_{n_0B_s \subset D}\frac{1}{|B_s|}\int_{B_s}|u(x)-u_{B_s}|\;\mathrm{d}m\right)
\end{align*}
for $v:= u \circ \tau_p \circ \delta_r$.
\end{proof}

\begin{thm}\label{t:BMO_BMO_loc}
The following statements hold true for every open set $\Om$ in $\mathbb{H}^1$ with constants independent of $\Omega$.
\begin{enumerate}
\item\label{i:BMO1} $\mathrm{BMO}^s(\Omega)=\mathrm{BMO}^s_{n,loc}(\Omega)$ {for all }$n>1$ with
\begin{equation}\label{eq:BMO_norm_s_loc}
 \|\cdot\|_{{\rm BMO}_{n,loc}^{s}(\Omega)} \leq \|\cdot\|_{\ast}^s \leq c_n^s  \|\cdot\|_{{\rm BMO}_{n,loc}^{s}(\Omega)},
\end{equation}
\item\label{i:BMO2} $\mathrm{BMO}^s(\Omega)=\mathrm{BMO}^{\mathbb{H}^1}(\Omega)$ with
\begin{equation}\label{eq:BMO_s_H_norms}
c_1 \|\cdot\|^{\mathbb{H}^1}_{\ast} \leq \|\cdot\|^s_{\ast} \leq c_2 \|\cdot\|^{\mathbb{H}^1}_{\ast}
\end{equation}
\item\label{i:BMO3} $\mathrm{BMO}^{\Hei}(\Omega)=\mathrm{BMO}^{\Hei}_{n,loc}(\Omega)$ {for all }$n>1$ with
\begin{displaymath}
 \|\cdot\|_{{\rm BMO}_{n,loc}^{\Hei}(\Omega)} \leq \|\cdot\|_{\ast}^{\Hei} \leq c_n^{\Hei}  \|\cdot\|_{{\rm BMO}_{n,loc}^{\Hei}(\Omega)}.
\end{displaymath}
\end{enumerate}
\end{thm}

\begin{proof}
%DELETED:
%Let us first observe that the following inclusion holds trivially for any $n>1$:
%ADDED:
First, the following holds trivially for any $n>1$:
$$
\mathrm{BMO}^{\Hei}(\Omega)\subseteq \mathrm{BMO}^{\Hei}_{n, loc}(\Omega)\quad\text{and}\quad \|\cdot\|_{{\rm BMO}_{n,loc}^{\Hei}(\Omega)} \leq \|\cdot\|_{\ast}^{\Hei}
$$
and the same  is true for the BMO-spaces considered with respect to the sub-Riemannian distance. Next, as claimed in \eqref{i:BMO1}, we note that
\begin{equation}\label{eq:equality_s_BMO}
\mathrm{BMO}^s(\Omega)=\mathrm{BMO}_{n,loc}^s(\Omega)\text{ for all }n>1,
\end{equation}
 where
 only the inclusion ``$\supseteq$'' is nontrivial.
In the setting of doubling metric measure spaces with a length metric, the assertion is essentially proven by Buckley in Theorem 0.3 in \cite{MR1749313}. However, Buckley's proof specifically considers balls $B$ such that $2B \subset \Omega$ and requires significant effort in its adaptation to the more general condition $nB \subset \Omega$ for some $n>2$. Instead we follow the shorter and more transparent proof of Theorem 2.2 in~\cite{maas}. As $\Hei$ equipped with the sub-Riemannian distance $d_s$ is a length space, Theorem 2.2 in~\cite{maas} can now be applied with $\lambda=n$, see the proof of~\cite[Theorem 2.2]{maas}. Since our reasoning is verbatim the same as in \cite{maas} for $n=2$, let us focus only on the modifications required for general $n$.

Consider a ball $B(x_0,R)$ that is admissible for $\|\cdot\|_{\ast}^s$ and a point $x\in B(x_0,R)$.
The key part of the proof is to construct a certain chain of balls $B_i(x_i,r_i)$ in $\Om$ centered on  a geodesic and connecting small enough balls $B(x_0,r_0)$ and $B(x,r_x)$ in such a way that  for $i=1,\ldots, N-1$ one has $r_i=\frac{R-d_s(x_0,x_i)}{2^{n+1}}$. Then one shows that $r_i=\alpha^{N-i}r_N$, with $\alpha:=\frac{2^{n+1}+1}{2^{n+1}-1}$. As in \cite[(2.11)]{maas}, we then obtain the following estimate for the length of the chain: $N-1\leq c\log_{\alpha}\left(\frac{R}{R-d_s(x_0,x)}\right)$. Thus,
%DELETED:
%a
%ADDED:
the
 counterpart of \cite[(2.12)]{maas} reads: $N-1\leq c\log_{\alpha}n^k$ if $x$ is at a certain distance of $x_0$, as quantified by $k$.
The rest of the proof  follows exactly as in \cite{maas}. This yields \eqref{eq:equality_s_BMO} with the quantitative control on the respective $\mathrm{BMO}$-norms stated in \eqref{eq:BMO_norm_s_loc} (depending on $n$ and the doubling constant of $(\Hei,d_s,m)$).

We now proceed to part \eqref{i:BMO2} of the theorem.
It follows from Lemma \ref{c:bmo-sub-Hei} and  \eqref{eq:equality_s_BMO} that
\begin{displaymath}
\mathrm{BMO}^{\mathbb{H}^1}(\Omega) \subseteq \mathrm{BMO}_{n,loc}^{\mathbb{H}^1}(\Omega) \subseteq \mathrm{BMO}_{\sqrt{\pi}n,loc}^s(\Omega)=\mathrm{BMO}^s(\Omega)
\end{displaymath}
for all $n>1$, with $\|\cdot\|_{\ast}^s$ bounded from above by a universal constant times $\|\cdot\|_{\ast}^{\mathbb{H}^1}$.
 The goal now is to show the reverse inclusion. To do so, we use the
 %DELETED:
% established
 fact that all Kor\'{a}nyi balls are $L^1$-averaging domains with uniform constants, cf. Corollary~\ref{lem-kor-ave}. Let $D_{\mathbb{H}^1}$ denote a Kor\'{a}nyi ball contained in $\Om$. Then
\begin{align*}
\|u\|_{{\rm BMO}^{\mathbb{H}^1}(\Omega)}&= \sup_{D_{\mathbb{H}^1}\subseteq \Omega}\frac{1}{|D_{\mathbb{H}^1}|}\int_{D_{\mathbb{H}^1}}|u-u_{D_{\mathbb{H}^1}}|\;\mathrm{d}m\\
&\leq \sup_{D_{\mathbb{H}^1}\subseteq \Omega} C_{ave}
\left(\sup_{nB_s \subset D_{\mathbb{H}^1}}\frac{1}{|B_s|}\int_{B_s}|u(x)-u_{B_s}|\;\mathrm{d}m\right)\\
&\leq  C_{ave} \sup_{B_s \subset \Omega} \frac{1}{|B_s|}\int_{B_s}|u(x)-u_{B_s}|\;\mathrm{d}m\\
&=  C_{ave} \|u\|_{{\rm BMO}^{s}(\Omega)}.
\end{align*}
This concludes the proof of \eqref{eq:BMO_s_H_norms}. Combined with this and \eqref{eq:equality_s_BMO}, part \eqref{i:BMO3} of the theorem  is now a direct consequence of Lemma~\ref{c:bmo-sub-Hei}.
\end{proof}

As mentioned in the introduction to Section~\ref{s:JacobianQC}, our goal is to show that  $\log J_f$ belongs to $\mathrm{BMO}(\Omega)$ (both with respect to $d_{\Hei}$ and $d_s$) for every quasiconformal map $f$ on the domain $\Omega$. In order to complete this goal, we need to recall and show a number of auxiliary results. This is done in Sections~\ref{s:revHol} and~\ref{s:weights}.

\subsection{Higher integrability and reverse H\"older inequality}\label{s:revHol}

Quasiconformal mappings between domains in $\mathbb{H}^n$ satisfy a higher integrability property analogous to the one established by F.\ Gehring in $\mathbb{R}^n$, see~\cite{MR0402038}. Specialized to the first Heisenberg group (endowed with the Kor\'{a}nyi distance), Theorem G in \cite{MR1317384} reads as follows.

\begin{thm}[Kor\'{a}nyi, Reimann]\label{t:GehringHeisenberg}
 For every $K\geq 1$, there exist constants $c>1$ and $\kappa>0$ depending only on $K$ such that
the Jacobian of a $K$-quasiconformal mapping $f: \Omega \to \Omega'$ between domains in $(\mathbb{H}^1,d_{\mathbb{H}^1})$ satisfies the inequality
\begin{equation}\label{eq:reverseHoelder}
\left(\Barint_{B(x,r)} J_f^{\frac{p}{4}} \;\mathrm{d}m\right)^{\frac{4}{p}} \leq \left(\frac{\kappa}{4+\kappa-p}\right)^{\frac{1}{p}} \Barint_{B(x,r)} J_f\;\mathrm{d}m,
\end{equation}
for all $p\in [4,4+\kappa)$ and $r\leq \frac{\mathrm{d}_{\mathbb{H}^1}(x,\partial \Omega)}{c}$.
\end{thm}

In this theorem, balls are defined with respect to $d_{\Hei}$. The statement of  Theorem G in~\cite{MR1317384} contains no mean value integrals, however the proofs of Propositions 19, 21 and 23 in~\cite{MR1317384} reveal that assertion \eqref{eq:reverseHoelder} holds as stated above. The proof is based on a Heisenberg version of Gehring's lemma and a reverse H\"older inequality for $J_f$.

\subsection{$A_p$-weights and distortion of balls}\label{s:weights} In this section we apply Theorem \ref{t:GehringHeisenberg} in order to show that the Jacobian of a quasiconformal map satisfies an $A_p$-weight  condition on balls which lie well inside the domain of the mapping, see Proposition \ref{p:A_p_Jacobian} for the precise statement. This is one of the key ingredients in the proof of the $\mathrm{BMO}$ property of the Jacobian (Theorem~\ref{c:logBMOdomain}).

\subsubsection{Distortion of balls}

For the proof of Proposition \ref{p:A_p_Jacobian}, we use the following auxiliary observation, which can be deduced from the local quasisymmetry property of quasiconformal mappings between domains in $\Hei$. In descriptive terms, it says that if a ball $B$ lies well enough inside the domain of a quasiconformal map $f$, then also its image lies well inside the image domain $\Omega'$ and moreover one can find an annulus around $B$ whose image looks again like a spherical annulus, where the statements can be quantified in terms of the quasiconformality constant of $f$ and the distance of $f(B)$ from $\partial \Omega'$.

\begin{proposition}\label{p:ball_dist}
For every $K\geq 1$, $c>1$, and $d\in \{d_s,d_{\Hei}\}$, there exists a constant $k>1$ such that the following holds. Whenever $f:\Omega \to \Omega'$ is a $K$-quasiconformal map between domains in $\Hei$, and $B=B(x,r)$ is a ball in $(\Hei,d)$ such that $B(x,10 k r)\subset \Omega$, then there exists a ball $B'=B(f(x),r')$ in $(\Hei,d)$ such that
\begin{enumerate}
\item\label{i:dist1} $f(B) \subset B'\subset f(kB)$,
\item\label{i:dist2} $cB' \subset \Omega'$.
\end{enumerate}
Moreover, every such ball $B'$ satisfies
\begin{equation}\label{eq:cor_ball_dist}
\mathrm{diam}(f(B))\leq \frac{2}{c-1} d(f(x),\partial \Omega')\quad \text{and}\quad d(f(x),\partial \Omega') \leq \frac{c}{c-1}  \mathrm{dist}(f(B),\partial \Omega').
\end{equation}
\end{proposition}

\begin{proof} For $k>1$ to be determined, let $B=B(x,r)$ be as in the assumptions of the proposition and set
\begin{displaymath}
r':= \sup_{\{u:\,d(x,u)=r\}} d(f(x),f(u)).
\end{displaymath}
We will verify that $B':= B(f(x),r')$ has properties (1) and (2) if $k$ is chosen sufficiently large (depending on $K$ and $c$). The inclusion $f(B)\subseteq B'$, and hence one half of the claim (1), is immediate from the choice of $r'$ and the path-connectedness of $\Hei\setminus B'$. Indeed, if there was $w\in B$ with $f(w)\in \Hei\setminus B'$, then by connecting $f(w)$ with a point in $\Hei \setminus f(kB)$ by a curve in $\Hei\setminus B'$, one could find $u'\in \partial B$ with $d(f(x),f(u'))>r'$, which contradicts the definition of $r'$.

Next, since $B(x,10 k r)\subset \Omega$, then by the discussion in Section~\ref{subs-2-3} we know that $f$ is $\eta$-quasisymmetric on $\overline{B(x,kr)}$ for a homeomorphism $\eta:[0,\infty) \to [0,\infty)$ that depends only on $K$. Let $y\in \partial B$ be such that
\begin{displaymath}
d(f(x),f(y))= r'
\end{displaymath}
and $z\in \partial B(x,kr)$ be such that
\begin{displaymath}
d(f(x),f(z))= \inf_{\{v:\; d(x,v)=kr\}} d(f(x),f(v)).
\end{displaymath}
Hence, by $\eta$-quasisymmetry,
\begin{equation}\label{eq:QS_applic}
\frac{d(f(x),f(y))}{d(f(x),f(z))}\leq \eta\left(\frac{d(x,y)}{d(x,z)}\right)= \eta\left(\frac{1}{k}\right).
\end{equation}
Since $\eta$ is an (increasing) homeomorphism, we can choose $k>1$ depending on $K$ and $c$ such that $\eta(1/k)\leq 1/c$. Moreover, as $\overline{k B} \subset \Omega$, we know by the choice of $z$ that $B(f(x),d(f(x),f(z))) \subset \Omega'$ and hence by \eqref{eq:QS_applic}
\begin{displaymath}
cr' =c d(f(x),f(y))\leq d(f(x),f(z)) \leq d(f(x),\partial \Omega'),
\end{displaymath}
so we have established \eqref{i:dist2}. Since the estimate \eqref{eq:QS_applic} for our choice of $k$ yields in particular that
\begin{displaymath}
r'= d(f(x),f(y))< d(f(x),f(z))
\end{displaymath}
and since $B(f(x),d(f(x),f(z)))\subset f(kB)$, as can be seen by the definition of $z$ and the path connectedness of $B(f(x),d(f(x),f(z)))$,
 we have also proven the second part of \eqref{i:dist1}, namely that $B'\subseteq f(kB)$. The proof of \eqref{i:dist1} and \eqref{i:dist2} is complete.

\medskip

The estimates in \eqref{eq:cor_ball_dist} are a straightforward consequence of (1) and (2). Indeed, given $B'=B(f(x),r')$ as in the first part of the proposition, we find
\begin{displaymath}
cr' \leq d(f(x),\partial \Omega') \leq r' + \mathrm{dist}(B',\partial \Omega').
\end{displaymath}
Hence,
\begin{displaymath}
\mathrm{diam}(f(B)) \leq \mathrm{diam} B'\leq \frac{2}{c-1} (c-1)r' \leq \frac{2}{c-1} \mathrm{dist}(B',\partial \Omega') \leq \frac{2}{c-1}d(f(x),\partial \Omega').
\end{displaymath}
Moreover,
\begin{displaymath}
d(f(x),\partial \Omega') \leq \mathrm{dist}(f(B),\partial \Omega') + r'\leq \mathrm{dist}(f(B),\partial \Omega') + \frac{1}{c} d(f(x),\partial \Omega'),
\end{displaymath}
which implies
\begin{displaymath}
d(f(x),\partial \Omega') \leq \frac{c}{c-1}  \mathrm{dist}(f(B),\partial \Omega').
\end{displaymath}
\end{proof}

\subsubsection{$A_p$-weights} With the preliminary ball distortion estimates at hand, we now prove the reverse H\"older property for the Jacobian of quasiconformal mappings on subdomains of the Heisenberg group endowed with the Kor\'{a}nyi distance.

\begin{proposition}\label{p:A_p_Jacobian}
Let $f:\Omega \to \Omega'$ be a $K$-quasiconformal map between domains in $(\mathbb{H}^1,d_{\mathbb{H}^1})$. Then there exist constants $C>0$ and $k>1$, depending on $K$ only, such that the weight condition
\begin{equation}\label{eq:weightJacobianHeis}
\Barint_B J_f \;\mathrm{d}m  \leq C \left(\Barint_B J_f^{-(p-4)/4}\;\mathrm{d}m\right)^{-\frac{4}{p-4}}
\end{equation}
holds for all $d_{\Hei}$-balls $B$ with $10kB\subset \Om$ and for all $p\in[4,4+\kappa)$.
\end{proposition}

 Here, $k$ denotes the constant from Proposition~\ref{p:ball_dist} if  $c$ is as in Theorem~\ref{t:GehringHeisenberg} applied to $f^{-1}$. Moreover, $\kappa$ is as in Theorem~\ref{t:GehringHeisenberg}.  In particular, the bound for $p$ depends on $K$ only.

\begin{proof} If $\Omega=\Omega'=\mathbb{H}^1$, the claim follows from Theorem~\ref{t:GehringHeisenberg} by estimates which are standard in the theory of Muckenhoupt weights, so we concentrate on the case where $\Om$ and $\Om'$ are strict subdomains in $\mathbb{H}^1$. While the proof still follows largely the same steps of reasoning as given in the proof of Lemma 5 in \cite{MR0361067} for a quasiconformal map from $\R^n$ to $\R^n$, we need to employ Proposition~\ref{p:ball_dist}
since our result is localized to sets $\Om$ and $\Om'$.

Let $c$ be the constant from Theorem~\ref{t:GehringHeisenberg} (applied to $\finv$) and $B'\subset \Om'$ be any ball such that $cB'\subset \Om'$. Since $J_{f^{-1}}(y)= J_f(f^{-1}(y))^{-1}$ for almost every $y\in \Om'$,
 Theorem~\ref{t:GehringHeisenberg} applied to $\finv$ shows that
\begin{equation}
\left(\Barint_{B'} J_{f}(\finv(y))^{-\frac{p}{4}}\,dy\right)^{\frac{4}{p}} \leq C \Barint_{B'} J_f(\finv(y))^{-1} \,dy, \label{ineq:invGehr}
\end{equation}
for $C= (\kappa/(4+\kappa-p))^{1/p}$.
 We are now in a position to apply Proposition~\ref{p:ball_dist} to $f$ with constant $c$ as defined in the beginning of the proof. Hence, we find a constant $k$ such that whenever the ball $B\subset \Om$ is such that $10kB\subset \Om$, then there exists a ball $B'\subset \Om'$ with the following properties:
 \begin{equation}
  cB'\subset \Om',\qquad  B\subset \finv(B'), \qquad \finv(B')\subset kB.\label{ball-Lemma4}
 \end{equation}
 In particular, \eqref{ineq:invGehr} applies to such $B'$. The inclusions \eqref{ball-Lemma4}
and the change of variable formula, see for instance Theorem 5.4(a) in~\cite{MR1778673}, result in the following estimate:
 \begin{equation*}
 \frac{1}{|B'|}\int_{B} J_{f}(x)^{-\frac{p}{4}}J_f(x)\,dx \leq
 \frac{1}{|B'|}\int_{\finv(B')} J_{f}(x)^{-\frac{p}{4}}J_f(x)\,dx =
 \Barint_{B'} J_{f}(\finv(y))^{-\frac{p}{4}}\,dy.
 \end{equation*}
 This and \eqref{ineq:invGehr} lead to the inequality
 \begin{equation}
 \left(\frac{1}{|B'|}\int_{B} J_{f}(x)^{-\frac{p}{4}}J_f(x)\,dx \right)^{\frac{4}{p}}\leq C \Barint_{B'} J_f(\finv(y))^{-1} \,dy = C\frac{|\finv(B')|}{|B'|}.\label{ineq2:invGehr}
 \end{equation}
 Since $|B'|=\int_{\finv(B')}J_f(x)\,dx$, we obtain from \eqref{ineq2:invGehr} that
 \[
  \left(\int_{B} J_{f}(x)^{\frac{4-p}{4}}\,dx \right)^{\frac{4}{p}}\leq C |\finv(B')|\,\left(\int_{\finv(B')}J_f(x)\,dx\right)^{\frac{4-p}{p}}.
 \]
By applying \eqref{ball-Lemma4} we get that $|\finv(B')|\leq |kB|=k^4|B|$. Thus, we arrive at the following inequality:
 \[
  \left(\int_{B} J_{f}(x)^{\frac{4-p}{4}}\,dx \right)^{\frac{4}{p}}\leq C k^4|B|^{\frac{4}{p}+\frac{p-4}{p}}\,\left(\int_{B}J_f(x)\,dx\right)^{\frac{4-p}{p}}.
 \]
 The above estimate immediately implies
 \[
  \left(\Barint_{B} J_{f}(x)^{\frac{4-p}{4}}\,dx \right)^{\frac{4}{p}}\leq Ck^4 \left(\Barint_{B}J_f(x)\,dx\right)^{\frac{4-p}{p}}
 \]
 and hence
 \[
  \Barint_{B}J_f(x)\,dx \leq \left(Ck^4\right)^{\frac{p}{p-4}} \left(\Barint_{B} J_{f}(x)^{\frac{4-p}{4}}\,dx \right)^{\frac{4}{4-p}}.
 \]
\end{proof}

Proposition \ref{p:A_p_Jacobian} essentially says that $w=J_f$ is a Muckenhoupt $A_q$-weight in the sense of \cite{MR3265363,MR0499948} for $q=1+4/(p-4)$.%\frac{4}{p-4}$.
This is not quite true, because the $A_q$-weight condition is verified only for those Kor\'{a}nyi balls $B$ which lie well inside the domain $\Omega$ in the sense that $10k B \subset \Omega$. It is for this reason that one cannot conclude that $J_f$ is an $A_{\infty}$-weight if $f$ is defined on an arbitrarily given domain $\Omega$ in $\mathbb{H}^1$. This also means that one cannot directly conclude that $\log J_f \in \mathrm{BMO}(\Omega)$, but as we will see in the next section, the latter statement still holds true.

\subsection{$\mathrm{BMO}(\Omega)$ property of $\log J_f$}\label{s:logJacobian-proof}
Lemma 3 in \cite{MR0361067} characterizes functions whose logarithm lies in $\mathrm{BMO}(\R^n)$ via an integral inequality.
We generalize part of this result to functions defined on a domain $\Om\subset \Hei$.
Since in our discussion we only need the implication in one direction, and only for functions which arise as quasiconformal Jacobians, we state the following result.

\begin{proposition}\label{p:Weight_log_BMO}
Let $\Omega$ be a domain in $\mathbb{H}^1$ and $f:\Om\to f(\Om)\subseteq\Hei$ be a $K$-quasiconformal mapping. Then
$\log J_f\in \mathrm{BMO}^{\Hei}_{loc}(\Omega)$ with $\|\log J_f\|_{\mathrm{BMO}_{10k,loc}^{\Hei}(\Omega)}$ bounded by a constant that depends only on $K$.
\end{proposition}

\begin{proof}
The proof is similar to the one for Lemma 3 in \cite{MR0361067} and, therefore, we omit the details. We start from the inequality
\begin{equation*}
\Barint_B J_f \;\mathrm{d}m  \leq C \left(\Barint_B J_f^{-(p-4)/4}\;\mathrm{d}m\right)^{-\frac{4}{p-4}},
\end{equation*}
which, by Proposition~\ref{p:A_p_Jacobian}, holds for all balls $B$ in $\Omega$ such that $10kB\subset \Om$, with $C$ and the bounds for $p$ determined by Theorem~\ref{t:GehringHeisenberg}. The crucial estimate, giving the $\mathrm{BMO}_{loc,10k}^{\mathbb{H}^1}(\Omega)$-norm bound, cf.~\cite[(2.7)]{MR0361067},  reads
\[
 \Barint_{B}\left|\log J_f-\left(\log J_f\right)_{B}\right|\;\mathrm{d}m\leq \frac{1}{s}\log(C^s+C^{-s})
\]
for $s=\min \{1,(p-4)/4\}$ and the constant $C$ depending on $K$, see the discussion of the constants in Theorem~\ref{t:GehringHeisenberg} and Proposition~\ref{p:A_p_Jacobian}.
\end{proof}

As a consequence of the above discussion, we deduce the main result of this section.
%\begin{thm}\label{c:logBMOdomain} The following holds both for the Kor\'{a}nyi and the sub-Riemannian distance:
%Let $f:\Omega \to \Omega'$ be a $K$-quasiconformal map between domains in $\Hei$. Then $\log J_f$ belongs to $%\mathrm{BMO}(\Omega)$ with a bound for $\|\log J_f\|_{\ast}$ in terms of $K$.
%\end{thm}
\begin{proof}[Proof of Theorem~\ref{c:logBMOdomain}]
 By Proposition~\ref{p:Weight_log_BMO}, we have that $\log J_f\in \mathrm{BMO}^{\mathbb{H}^1}_{loc}(\Omega)$ with a quantitative upper bound for its $\mathrm{BMO}_{10k,loc}^{\mathbb{H}^1}(\Omega)$-norm that depends on $K$ only. Theorem~\ref{t:BMO_BMO_loc} allows us to conclude that $\log J_f \in \mathrm{BMO}(\Omega)$ and to bound its $\mathrm{BMO}(\Omega)$-norm (both with respect to
 the Kor\'{a}nyi and the sub-Riemannian metric) from above in terms of $K$.
\end{proof}

%%%%% Section 4
%%%%%

\section{A Koebe theorem}\label{s:Koebe}

The main purpose of this section is to prove the key result of this work: the Koebe theorem for quasiconformal mappings between open sets in $\Hei$, see Theorem~\ref{t:KoebeHeis}. Before providing the proof of this result we need further auxiliary observations.

The following lemma is a counterpart of \cite[Lemma 5.10]{MR861687}.
For our purposes it suffices to consider balls centered at one point, but we consider arbitrary domains in $\mathbb{H}^1$ instead of disks in $\mathbb{R}^2$.
The proof goes the same way, and so we omit it.

\begin{lemma}\label{l:ball_ratio} The following statement holds both for the Kor\'{a}nyi distance $d_{\mathbb{H}^1}$ and the sub-Riemannian distance $d_s$:
Let $\Omega$ be a domain in $\mathbb{H}^1$ and let $u\in \mathrm{BMO}(\Omega)$. Then for all balls $B_2 \subset B_1 \subset \Omega$ centered at a point $z\in \Omega$, one has
\begin{equation}
|u_{B_1}-u_{B_2}|\leq \frac{e}{2} \left(\log \frac{|B_1|}{|B_2|}+1\right) \|u\|_{\ast}.  \label{BmoMeanL}
\end{equation}
\end{lemma}
Here and in the following, the logarithm is taken with respect to the basis $e$.

In the literature, the definition of the quantity $a_f(x)$ for a $K$-quasiconformal map $f:\Omega \to \Omega'$ between domains $\Omega,\Omega'  \subsetneq \mathbb{R}^n$ involves taking averages of $\log J_f$ over either $B(x,d(x,\partial \Omega))$ or over $B(x,d(x,\partial \Omega)/2)$. It turns out that the resulting quantities are comparable. In $\mathbb{H}^1$, it is for technical reasons sometimes more convenient to work with $B(x,d(x,\partial \Omega)/L)$ for a number $1<L<\infty$ which depends only on $K$ (for instance in the proof of Proposition \ref{c:dist_est} below). The following lemma shows that this is possible.

\begin{lemma}\label{l:a_f_comparison} Let $d\in \{d_s,d_{\mathbb{H}^1}\}$ and $K\geq 1$. Given $1\leq L <\infty$, there exists  a constant $1\leq C_{K,L}<\infty$, such that for every $K$-quasiconformal mapping $f:\Omega \to \Omega'$ between domains $\Omega,\Omega'  \subsetneq (\mathbb{H}^1,d)$, one has
\begin{displaymath}
C_{K,L}^{-1} \exp\left(\tfrac{1}{4}\left(\log J_f\right)_{B_1}\right) \leq
\exp\left(\tfrac{1}{4}\left(\log J_f\right)_{B_2}\right) \leq C_{K,L}
\exp\left(\tfrac{1}{4}\left(\log J_f\right)_{B_1}\right),
\end{displaymath}
where
\begin{displaymath}
B_1:=B(x,d(x,\partial \Omega))\quad\text{and}\quad B_2:= B(x,d(x,\partial \Omega)/L).
\end{displaymath}
\end{lemma}

\begin{proof} If $u\in \mathrm{BMO}(\Omega)$, then for a given $1<L<\infty$, Ahlfors regularity of the measure $m$ and \eqref{BmoMeanL} show that
\begin{equation}\label{eq:BMO_estimate}
|u_{B_1}-u_{B_2}|\leq c_1(L) \|u\|_{\ast}, \quad \text{where}\quad  c(L)= \frac{e}{2} (4 \log(L)+1).
\end{equation}

{\color{red} }
By Theorem~\ref{c:logBMOdomain}, if $f$ is $K$-quasiconformal, then $\log J_f \in \mathrm{BMO}(\Omega)$ and $\|\log J_f\|_{\ast}$ is bounded by a constant $c_2(K)$ depending on $K$ only. Hence, for $u=\log J_f$,  \eqref{eq:BMO_estimate} gives
\begin{align*}
\exp\left(\tfrac{1}{4}\left(\log J_f\right)_{B_2}\right)& \leq \exp\left(\tfrac{1}{4}\left(\log J_f\right)_{B_1}+ c_{K,L} \right)=C_{K,L} \exp\left(\tfrac{1}{4}\left(\log J_f\right)_{B_1}\right)
\end{align*}
and
\begin{align*}
\exp\left(\tfrac{1}{4}\left(\log J_f\right)_{B_1}\right)& \leq \exp\left(\tfrac{1}{4}\left(\log J_f\right)_{B_2}+ c_{K,L}\right)=C_{K,L} \exp\left(\tfrac{1}{4}\left(\log J_f\right)_{B_2}\right)
\end{align*}
 where $c_{K,L}=\frac{1}{4}c_1(L)c_2(K)$ and $C_{K,L}:= \exp(c_{K,L})$. The statement of the lemma follows.
\end{proof}

%
%We make use of an ``egg yolk'' principle as established in \cite[Lemma 5.2]{MR3029176} (for domains with two-point boundary condition in geodesic spaces of globally $Q$-bounded geometry). Since $(\mathbb{H}^1,d_s)$ is unbounded, of globally $4$-bounded geometry and  geodesic when equipped with the sub-Riemannian distance, the yolk principal holds and reads as follows:
%
%\begin{proposition}[Egg yolk principle]\label{p:egg_yolk}
%Let $f:\Omega \to  \Omega'$ be a $K$-quasiconformal mapping between domains $\Omega,\Omega'  \subsetneq (\mathbb{H}^1,d_s)$. Then there exists a universal constant $0<\kappa\leq 1$ such that, for all $x\in \Omega$, the map $f$ restricted to the ball $B(x,d_s(x,\partial \Omega)/(4\kappa +1))$ is quasisymmetric, with the quasisymmetry data depending only on $K$ and the data of $\mathbb{H}^1$.
%\end{proposition}
%
%\begin{remark}\label{r:egg_yolk_koranyi}
%As a special case of \cite[Lemma 5.2]{MR3029176}, Proposition \ref{p:egg_yolk} holds a priori with respect to the sub-Riemannian distance $d=d_s$. Since $d_s$ is bi-Lipschitz equivalent to the Kor\'{a}nyi distance, it follows immediately that an analogous statement (with quantitatively comparable quasisymmetry data) also holds for $d_{\mathbb{H}^1}$, however for a different constant $\kappa$ which is still universal but might be larger than $1$.
%\end{remark}
%

The following proposition is due to Soultanis and Williams, \cite[Corollary 5.3]{MR3029176}, for domains with two-point boundary condition in geodesic metric measure spaces of globally $Q$-bounded geometry ($Q>1$). (See also \cite[p.627]{MR3029176} for a remark on these assumptions.)  Since $(\mathbb{H}^1,d_s,m)$ is unbounded, of globally $4$-bounded geometry and  geodesic, their result applies in our setting. The corresponding result for the metric $d_{\Hei}$ instead of $d_s$ then follows immediately from the comparability of the metrics, more precisely from \eqref{eq:Whitney_ball_inclusion}. The constant ``$10$'' in the statement of the proposition is not optimal, but convenient since it works simultaneously for $d_s$ and $d_{\Hei}$.

\begin{proposition}\label{p:egg_yolk_cor} Let $d\in \{d_s,d_{\Hei}\}$ and $K\geq 1$. Then there exists a homeomorphism $\eta: [0,+\infty) \to [0,+\infty)$ such that the following holds for all $K$-quasiconformal maps $f:\Omega \to \Omega'$ between domains in $\Hei$ and for all $x\in \Omega$. If $y\in \Omega$ satisfies $d(y,x) \leq d(x,\partial \Omega)/10$, then
\begin{displaymath}
\frac{d(f(x),f(y))}{d(f(x),\partial \Omega')}\leq \eta \left(\frac{d(x,y)}{d(x,\partial \Omega)}\right).
\end{displaymath}
\end{proposition}

The proposition exploits the connectedness and doubling property of balls in $(\Hei,d)$ and is a stronger statement than what one can derive merely based on local $\eta$-quasisymmetry in arbitrary metric spaces. The local quasisymmetry of quasiconformal maps coupled with the reverse H\"older property for their Jacobians also yields the following distance estimate.

\begin{proposition}\label{c:dist_est}
For every $1\leq K< \infty$, there exists a constant $\lambda\geq 1$ such that the following holds.
If $0<a<b$ and $f:\Omega \to \Omega'$ is a $K$-quasiconformal mapping between domains $\Omega,\Omega' \subsetneq \mathbb{H}^1$, and $\|\log J_f\|_{\ast} \leq (2/e) a$, then for all $z_1,z_2\in \Omega$ such that
\begin{equation*}
d_{\mathbb{H}^1}(z_1,z_2)\leq \tfrac{1}{2\lambda} d_{\mathbb{H}^1}(z_1,\partial \Omega),
\end{equation*}
one has
\begin{equation}\label{eq:distance_est}
d_{\mathbb{H}^1}(f(z_1),f(z_2))\leq c a_f(z_1) d_{\mathbb{H}^1}(z_1,\partial \Omega)^a d_{\mathbb{H}^1}(z_1,z_2)^{1-a},
\end{equation}
where the constant $c$ depends only on $K$ and the bound $b$ for $a$.
\end{proposition}

The proof of Proposition \ref{c:dist_est} will show that $c$ in the statement can be obtained as a monotone increasing function of $b$. Also, for large values of $a$, the conclusion in \eqref{eq:distance_est} is not very informative: assume for illustrative purposes  that $d_{\mathbb{H}^1}(z_1,z_2)= \tfrac{1}{2\lambda} d_{\mathbb{H}^1}(z_1,\partial \Omega)$. Then the right-hand side of \eqref{eq:distance_est} is comparable to $\lambda^{a-1} d_{\mathbb{H}^1}(z_1,\partial \Omega)$. If $a$ becomes large, then also the multiplicative constant in this upper bound tends to infinity.
For these reasons, in our application we will be interested in having a small upper bound $b$ for $a$.

Proposition \ref{c:dist_est} is a counterpart for Lemma 5.15 in \cite{MR861687}, for arbitrary subdomains of $\mathbb{H}^1$ instead of disks in $\mathbb{R}^2$.
A significant difference in the proof arises from the fact that (i) we do not know whether the map $f$ extends to a $K_1(K)$-quasiconformal map on the one point compactification of $\mathbb{H}^1$, and (ii) we do not have a Mori distortion theorem at our disposal. We compensate for this by resorting to the local quasisymmetry property of $f$. This, and the ball admissibility condition for \eqref{eq:weightJacobianHeis}, accounts for the presence of the constant $\lambda$ in the statement of Proposition \ref{c:dist_est}. Lemma 5.15 in \cite{MR861687} contains an analogous statement in $\mathbb{R}^2$ with $\lambda=1$. The weaker formulation of Proposition \ref{c:dist_est} influences the proof of Koebe's theorem, where now $r_2\leq r_1/(2\lambda)$ has to be used instead of $r_2=r_1/2$.
In light of Lemma  \ref{l:a_f_comparison}, this change is immaterial. Moreover, working with arbitrary subdomains, rather than just disks or balls,  has the advantage that we only have to define $a_f$ and $\|\cdot\|_{\ast}$ for one domain, namely $\Omega$. Also note that if $z\in B\subset \Omega$, then $d(z,\partial B)\leq d(z,\partial \Omega)$, and for our purpose an estimate in terms of $d(z,\partial \Omega)$ is sufficient.

\begin{proof}[Proof of Proposition~\ref{c:dist_est}] Throughout the proof we will work with the Kor\'{a}nyi distance $d=d_{\mathbb{H}^1}$.
The idea is to choose $\lambda$ large enough so that for $z_1$ and $z_2$ as in the assumptions, we have that:
\begin{enumerate}
\item $f$ is quasisymmetric on $\overline{B(z_1,d(z_1,z_2))}$,
\item $B(z_1,2d(z_1,z_2))$ is admissible for the reverse H\"older inequality for $J_f$ \eqref{eq:weightJacobianHeis}.
\end{enumerate}

The requirements are satisfied under the following assumptions:
\begin{enumerate}
\item\label{i:lambda1} $\lambda \geq 5$
 (egg yolk principle for $d_{\mathbb{H}^1}$),
\item\label{i:lambda2} $\lambda \geq 10k$, where $k$ is as in the admissibility condition for  \eqref{eq:weightJacobianHeis}.
\end{enumerate}

Let us choose $\lambda\geq 1$ as the smallest constant for which \eqref{i:lambda1} and \eqref{i:lambda2} hold. Such a $\lambda$ is finite and depends only on $K$. We set
\begin{displaymath}
r_1:= d(z_1,\partial \Omega)/\lambda\quad\text{and}\quad r_2:= d(z_1,z_2)
\end{displaymath}
and denote $B_i:=B(z_1,r_i)$ for $i\in \{1,2\}$. Moreover, we write
\begin{displaymath}
s_2:= \min_{d(z,z_1)=r_2}d(f(z),f(z_1)).
\end{displaymath}
Since $f|_{\overline{B(z_1,d(z_1,z_2))}}$ is $H$-quasisymmetric for a constant $H$ which depends only on $K$, we find
\begin{displaymath}
d(f(z_1),f(z_2))\leq H s_2.
\end{displaymath}
This implies
\begin{equation}\label{eq:quotient_ineq}
\left(\frac{d(f(z_1),f(z_2))}{d(z_1,z_2)}\right)^4 \leq \left(\frac{H s_2}{r_2}\right)^4
\leq H^4 \frac{|f(B_2)|}{|B_2|}.
\end{equation}
In order to further estimate this from above, an analog of Lemma 5.14 in \cite{MR861687} would be useful. We concentrate on a manageable special case, which is sufficient for our application. Namely, using the same notation as above, we will prove that
\begin{equation}\label{eq:5.14}
\frac{|f(B_2)|}{|B_2|}\leq C' \exp \left(\frac{1}{|B_2|}\int_{B_2}\log J_f\;\mathrm{d}m\right),
\end{equation}
where $C'$ is a constant depending only on $K$. To show this, we consider the enlarged ball
\begin{displaymath}
Q:= 2 B_2:= B(z_1,2d(z_1,z_2)).
\end{displaymath}
By our choice of $\lambda$, the ball $Q$ is admissible for  \eqref{eq:weightJacobianHeis}, and we deduce  for a suitable constant $c_1>0$ (depending on $K$), that
\begin{equation}\label{eq:est_meas_ratio}
\frac{|f(Q)|}{|Q|}\leq c_1 \left(\frac{1}{|Q|} \int_Q J_f^{-(p-4)/4}\;\mathrm{d}m\right)^{-4/(p-4)}\\
\leq c_1 \exp\left(\frac{1}{|Q|}\int_Q \log J_f \;\mathrm{d}m\right),
\end{equation}
where we have applied Jensen's inequality to the convex function $\varphi(x)=e^{-bx}$ for $b=(p-4)/4$ in the last step.
The remaining steps to deduce \eqref{eq:5.14} consist of a computations analogous to the proof of \cite[Lemma 5.14]{MR861687}. First, we observe that
\begin{displaymath}
\Barint_Q  \log J_f \;\mathrm{d}m = \frac{1}{2^4} \Barint_{B_2} \log J_f \;\mathrm{d}m +
\left(1-\frac{1}{2^4}\right) \Barint_{Q\setminus B_2} \log J_f\;\mathrm{d}m
\end{displaymath}
and
\begin{displaymath}
\Barint_{Q\setminus B_2} \log J_f\;\mathrm{d}m \leq \log \left(\Barint_{Q\setminus B_2} J_f\;\mathrm{d}m\right) =
\log \left(\frac{2^4}{2^4-1}\Barint_Q J_f\;\mathrm{d}m\right)\leq
\log \left(\frac{2^4}{2^4-1}\frac{|f(Q)|}{|Q|}\right).
\end{displaymath}
Inserting these estimates in \eqref{eq:est_meas_ratio}, we find that
\begin{displaymath}
\frac{1}{2^4} \log \frac{|f(Q)|}{|Q|}\leq \log c_1 +  \frac{1}{2^4}\Barint_{B_2}\log J_f \;\mathrm{d}m + \frac{2^4-1}{2^4}\log \frac{2^4}{2^4-1},
\end{displaymath}
which yields \eqref{eq:5.14} since $|f(B_2)|\leq |f(Q)|$ and $|Q|=2^4 |B_2|$.

Combining \eqref{eq:quotient_ineq} with \eqref{eq:5.14}, we deduce that
\begin{equation}\label{eq:dist_est}
\left(\frac{d(f(z_1),f(z_2))}{d(z_1,z_2)}\right)^4 \leq c_2 \exp \left((\log J_f)_{B_2}\right)
\end{equation}
for a constant $c_2$ depending only on $K$.
The rest of the argument is similar to the proof of \cite[Lemma 5.15]{MR861687}, with the factor $1/2$ replaced by $1/4$.
By Lemma \ref{l:ball_ratio} applied to $u=\log J_f$, we have
\begin{displaymath}
\left| (\log J_f)_{B_1}-(\log J_f)_{B_2} \right|\leq \left(\log \frac{|B_1|}{|B_2|}+1\right) a
= 4a \log \frac{r_1}{r_2} +a.
\end{displaymath}
Hence
\begin{displaymath}
\frac{1}{4} (\log J_f)_{B_2} \leq \frac{1}{4} (\log J_f)_{B_1}  +  a \log \frac{r_1}{r_2}+ \frac{b}{4}.
\end{displaymath}
Combined with \eqref{eq:dist_est}, this shows that
\begin{align*}
d(f(z_1),f(z_2))&\leq c_2^{1/4}\exp\left(\frac{1}{4}(\log J_f)_{B_2}\right)d(z_1,z_2) \\
& \leq c_2^{1/4} \exp(b/4) \exp\left(\frac{1}{4}(\log J_f)_{B_1}\right)d(z_1,\partial \Omega)^a d(z_1,z_2)^{1-a},
\end{align*}
which concludes the proof of the proposition.
\end{proof}

Finally we are in a position to prove the Koebe type theorem.

\begin{proof}[Proof of Theorem~\ref{t:KoebeHeis}]
As remarked in the introduction, it suffices to prove the theorem for the Kor\'{a}nyi distance $d=d_{\mathbb{H}^1}$, as then the corresponding statement for $d_s$ follows.

Let us first observe that the assumption $\Om\subsetneq \Hei$ implies that $\Om'=f(\Om)\subsetneq \Hei$.
Indeed, this is a consequence of the fact that $\Hei$ is not quasiconformally equivalent to any proper subdomain of $\Hei$ by Theorem 13.1 in~\cite{tys2001}. The result in \cite{tys2001} is formulated for  $Q$-Loewner spaces and locally quasisymmetric embeddings, hence it applies in particular in our setting.

We fix an arbitrary point $x_1\in \Omega$ and prove estimate \eqref{eq:KoebeEst} for $x=x_1$. To this end, we define
\begin{displaymath}
r_1:=d(x_1,\partial \Omega)\quad\text{and}\quad d_1:= d(f(x_1),\partial\Omega').
\end{displaymath}
Note that both $r_1, d_1\not=\infty$, as $\Om, \Om'\subsetneq \Hei$.

Set further
\begin{displaymath}
r_2:= r_1/m \quad\text{and}\quad d_2:=\max_{d(x_1,x)=r_2}d(f(x_1),f(x)),
\end{displaymath}
where $m= \max\{ 2\lambda, 10 k \}$ with $k$ as in Proposition \ref{p:ball_dist} applied with $c=10$, and $\lambda$ as in Proposition \ref{c:dist_est}.  Let further $x_2\in \Omega$ be a point which realizes the maximum in $d_2$, that is $d(x_1,x_2)=r_2$ and $d(f(x_1),f(x_2))=d_2$.

We denote
\begin{displaymath}
B_1:=B(x_1,r_1)\quad\text{and}\quad B_2:=B(x_1,r_2).
\end{displaymath}

We will show that there exists a positive and finite constant $c_1=c_1(K)$ such that
\begin{equation}\label{eq:Step1_main_proof}
\frac{1}{c_1} \leq \frac{d_1}{d_2} \leq c_1.
\end{equation}
The first inequality will be obtained by applying Proposition \ref{p:egg_yolk_cor} to $f$ and for this application it would suffice to know that $m\geq 10$. However, to derive the upper bound in \eqref{eq:Step1_main_proof}, we will have to apply  Proposition \ref{p:egg_yolk_cor} to the inverse map $f^{-1}$, and for this we need to know that the \emph{image} of $B_2$ is still contained in a ball that is admissible for Proposition \ref{p:egg_yolk_cor}. This is why we require that $m\geq 10 k$. The assumption $m\geq 2\lambda$ is used only in the second part of the theorem.

Proposition \ref{p:egg_yolk_cor}  applied to $f$, $x=x_1$ and $y=x_2$ yields
\begin{displaymath}
\frac{d(f(x_1),f(x_2))}{d(f(x_1), \partial \Omega')}\leq \eta\left(\frac{d(x_1,x_2)}{d(x_1,\partial \Omega)}\right)
\end{displaymath}
and hence, by the choice of $x_2$,
\begin{displaymath}
\frac{d_2}{d_1}\leq \eta\left(\frac{r_2}{r_1}\right)= \eta\left(\frac{1}{m}\right).
\end{displaymath}

To obtain the second bound in \eqref{eq:Step1_main_proof}, we apply  Proposition \ref{p:egg_yolk_cor}   to the ($K$-quasiconformal) map $f^{-1}$, $x=f(x_1)$ and $y=f(x_2)$. Here we note that $m\geq 10 k$ and $k$ has been chosen such that
\begin{displaymath}
f(x_2) \in f(B_2) \subset B'
\end{displaymath}
for some ball $B'$ centered at $f(x_1)$ with the property that $10 B'\subset \Omega'$. Hence Proposition \ref{p:egg_yolk_cor} is indeed applicable for the points $f(x_1)$ and $f(x_2)$ and we find
\begin{displaymath}
\frac{d(x_1,x_2)}{d(x_1,\partial \Omega)}\leq \eta\left(\frac{d(f(x_1),f(x_2))}{d(f(x_1),\partial \Omega')}\right).
\end{displaymath}
This yields the following bound:
\begin{displaymath}
\eta^{-1}\left(\frac{1}{m}\right)=\eta^{-1}\left(\frac{r_2}{r_1}\right)\leq \frac{d_2}{d_1}.
\end{displaymath}

%We are now able to deduce an upper bound for $a_f(x_1)$, analogously to \cite[(2.7)]{MR777305}. Indeed, since $f(B_2)\subset B(f(x_1),d_2)$, we find by Jensen's inequality and \eqref{eq:Step1_main_proof} that
%\begin{align*}
%(\log J_f)_{B_2}\leq \log \left(\frac{|f(B_2)|}{|B_2|}\right)\leq 4 \log \frac{d_2}{r_2}\leq 4 \left(\log \frac{m c_1 d_1}{r_1}\right).
%\end{align*}
%
%Combining this with Theorem \ref{c:logBMOdomain} and Lemma \ref{l:ball_ratio} (applied to $u=J_f$), we find
%\begin{displaymath}
%(\log J_f)_{B_1}\leq 4 \left(\log\left(\frac{m c_1 d_1}{r_1}\right) + c_2\right),
%\end{displaymath}
%for a constant $c_2$ which depends on $K$ only. Thus,
%\begin{equation}\label{eq:upperbound_af}
%a_f(x_1) \leq \frac{m c_1 d_1}{r_1}\exp(c_2).
%\end{equation}
%

We are now able to deduce an upper bound for $a_f(x_1)$, analogously to \cite[(2.7)]{MR777305}. Indeed, since $f(B_2)\subset B(f(x_1),d_2)$, we find by Jensen's inequality and \eqref{eq:Step1_main_proof} that
\begin{align*}
(\log J_f)_{B_2}\leq \log \left(\frac{|f(B_2)|}{|B_2|}\right)\leq 4 \log \frac{d_2}{r_2}\leq 4 \left(\log \frac{m c_1 d_1}{r_1}\right).
\end{align*}

Combining this with Theorem \ref{c:logBMOdomain} and Lemma \ref{l:ball_ratio} (applied to $u=J_f$), we find
\begin{displaymath}
(\log J_f)_{B_1}\leq 4 \left(\log\left(\frac{m c_1 d_1}{r_1}\right) + c_2\right),
\end{displaymath}
for a constant $c_2$ which depends on $K$ only. Thus,
\begin{equation}\label{eq:upperbound_af}
a_f(x_1) \leq \frac{m c_1 d_1}{r_1}\exp(c_2).
\end{equation}

The next step is to apply Proposition \ref{c:dist_est} for $z_1=x_1$ and $z_2=x_2$ in order to find a lower bound for $a_f(x_1)$. In this way we obtain, for constants $a$ and $c$ bounded in terms of $K$, that
\begin{align*}
d_2= d(f(x_1),f(x_2))\leq c a_f(x_1) d(x_1,\partial \Omega)^a d(x_1,x_2)^{1-a}\leq \left(\frac{1}{m}\right)^{1-a}c a_f(x_1)r_1.
\end{align*}
By \eqref{eq:Step1_main_proof} we can bound $d_2$ from below by $d_1/c_1$, which yields the desired lower bound for $a_f(x_1)$ and thus, together with the upper bound in \eqref{eq:upperbound_af}, concludes the proof.

\end{proof}

\section{Applications}\label{s:applications}

In this section we discuss applications of Theorem \ref{t:KoebeHeis}. Section \ref{subs:comp} contains analytic results regarding the horizontal derivative of a quasiconformal mapping. Results in the spirit of the ones in Sections \ref{ss:diam} and \ref{sec:metrics} have been obtained by H.\ Len Ruth Jr.\ in his PhD thesis, \cite[Section 3.7]{LRj}, for quasisymmetric mappings in a more abstract setting and for the quantity $d(f(\cdot),\partial \Omega')/d(\cdot,\partial \Omega)$. By our version of the Koebe theorem, $d(f(\cdot),\partial \Omega')/d(\cdot,\partial \Omega)$ is comparable to $a_f$ for quasiconformal mappings between domains in $\mathbb{H}^1$. In this sense, Lemma 3.7.4 and Proposition 3.7.5 in \cite{LRj} are quasisymmetric counterparts of our Propositions~\ref{t:diam} and~\ref{p:a_f_conf_density}, respectively. Since a quasiconformal map $f$ on a subdomain $\Omega \subset \mathbb{H}^1$ is in general only locally quasisymmetric, our results do not follow directly from the ones in \cite{LRj}.
%DELETED:
%Potential localization arguments are complicated by the fact that Propositions~\ref{t:diam} and~\ref{p:a_f_conf_density} (part (2)), concern the global behavior of the mapping. For this reason,
We give direct proofs in the quasiconformal category, which do not rely on the results in \cite{LRj}, but on similar proof arguments. The specific setting of the sub-Riemannian Heisenberg group allows us to illustrate the sharpness of Proposition~\ref{t:diam}  with an example and to formulate our results with less additional assumptions than in~\cite[Section 3.7]{LRj}. In particular, Proposition~\ref{p:a_f_conf_density} holds for arbitrary, not necessarily Ahlfors regular, domains.
%ADDED:
It shows that $a_f$ is a conformal density on $\Omega$, which is useful information, for instance in light of the results in \cite{MR2040578}.

\subsection{Diameter bounds for image curves}\label{ss:diam}

%DELETED:
%If $\gamma:[a,b]\to \mathbb{R}^n$ is a rectifiable curve, and $f:\mathbb{R}^n \to \mathbb{R}^n$ a smooth mapping, then one has the following bound for the Euclidean length of the image curve:
%\begin{displaymath}
%\mathrm{length}(f\circ \gamma) \leq \int_{\gamma} \|Df\| \;\mathrm{d}s,
%\end{displaymath}
%where $\mathrm{d}s$ denotes integration with respect to an arc length, and $\|\cdot\|$ is the operator norm (maximal stretch) of a linear map.
%An inequality of this form continues to hold if $\mathbb{R}^n$ is replaced by a general metric space, $f$ is assumed to be a locally Lipschitz mapping of $X$, and instead of  $\|Df\|$ one considers the pointwise lower Lipschitz constant, see \cite[Lemma 6.2.6]{MR3363168}.
%While quasiconformal mappings
%on nice enough metric spaces
% are absolutely continuous along almost every curve, one cannot hope to control for {every} rectifiable curve the length of its image in a similar manner. The reason is essentially that quasiconformality is not a condition on distances, but rather on ratios of distances.
%However, it has been shown by Koskela in  \cite[Lemma 2.6]{Koskela1994},

%ADDED:
A quasiconformal map can wildly distort the length of an individual curve.
However, Koskela has shown in  \cite[Lemma 2.6]{Koskela1994}
for quasiconformal mappings defined on a domain $\Omega$ in  $\mathbb{R}^n$, that  it is possible to control the \emph{diameter} of $f\circ \gamma$ in terms of $\int_{\gamma} a_f \; \mathrm{d}s$ for all curves $\gamma$ which are long enough in terms of their distance to the boundary of $\Omega$. The goal of this section is to study similar estimates in the Heisenberg group. Koskela's proof makes use of the Besicovitch covering theorem, which does not hold for the Kor\'anyi or the sub-Riemannian distance on $\mathbb{H}^1$. A possible approach would be to use one of the comparable distances with the Besicovitch covering property that were constructed in \cite{LeDonRigot}. Instead, we will give below a direct proof using the basic $3r$-covering theorem, which can be found for instance in
\cite[Theorem 2.1]{MR3154530}.
Our statement is slightly more flexible than the original version since we allow for a quantitative control of the lengths of curves in terms of a parameter $\alpha$; this will prove useful later in applications.

\begin{proposition}\label{t:diam}

 Let $d$ denote either the Kor\'anyi or sub-Riemannian distance on
$\mathbb{H}^1$. Let $f:\Omega \to \Omega'$ be
a $K$-quasiconformal mapping between domains in $\mathbb{H}^1$ with $\Omega \neq \mathbb{H}^1$.
Then, for every $\alpha\in(0,1]$ and for every rectifiable curve
$\gamma$ contained in $\Omega$ with
    \begin{equation}\label{eq:curve_length_cond}
    \mathrm{length}(\gamma)\geq \alpha d(\gamma,\partial \Omega),
    \end{equation}
    one has
    \begin{displaymath}
    \mathrm{diam}(f\circ \gamma) \leq C\int_{\gamma} a_f\;\mathrm{d}s
    \end{displaymath}
    for a constant $C$ which depends only on $\alpha$, $d$ and $K$. Here $\int \mathrm{d}s$ denotes integration with respect to the $d$-length.
\end{proposition}

Recall that curves have the same length with respect to $d_s$ and $d_{\mathbb{H}^1}$.
%DELETED:
%In particular, the two metrics generate the same class of rectifiable curves. It suffices to prove Proposition \ref{t:diam} for $d=d_{\mathbb{H}^1}$.
Since $\frac{1}{\sqrt{\pi}}d_s \leq d_{\mathbb{H}^1}\leq d_s$ and $a_f^s \simeq a_f^{\mathbb{H}^1}$,
%DELETED:
%the claim for $d_s$ follows from the result for $d_{\mathbb{H}^1}$.
%ADDED:
it suffices therefore  to prove Proposition \ref{t:diam} for, say,  $d=d_{\mathbb{H}^1}$.

\begin{proof}

Let $d=d_{\mathbb{H}^1}$.
    The following abbreviating notation will be used in this proof. For $0<\chi<1$, we denote
    $$
    B_{\chi}(x)={B(x,\chi d(x,\partial \Omega))}
    $$
    for $x \in \Omega$.

Let $k=k(2,K)$ be the constant given by
Proposition \ref{p:ball_dist} applied to $c=2$. Then we fix $\lambda\in (0,1)$ to be the largest number satisfying
\begin{equation}\label{eq:condLambdaDiam}
\lambda \leq \frac{3\alpha}{1+\alpha}\quad\text{and}\quad 10k
\lambda  \leq \frac{1}{2}.
\end{equation}
By the second condition we have that $10 k
B_{\lambda }(x) \subset \Omega$ for all $x \in \Omega$, which will be used to apply Proposition \ref{p:ball_dist}.
Furthermore, since $ k >1$, we have $\lambda \leq
\frac{1}{20k} \leq \frac{1}{20}$. The use of the first condition in \eqref{eq:condLambdaDiam} will become clear later.

Consider now an arbitrary curve $\gamma$ satisfying the assumptions of the proposition. For simplicity we continue to denote the trace of $\gamma$ by the symbol $\gamma$.
    Let $\mathcal{B}^\lambda$ be
    %DELETED
   % the family of
   %ADDED
   the cover of $\gamma$ given by
   all balls of the form $B_{\lambda/3}(p)$ where $p \in \gamma$.
   %DELETED
% The balls in this family cover $\gamma$.
By compactness of $\gamma$, we may without loss of generality assume that the family $\mathcal{B}^\lambda$ contains only finitely many balls.

 This allows us to apply the $3r$-covering lemma and select a (finite) disjointed subfamily  $ \mathcal{F}^\lambda \subset \mathcal{B}^\lambda$ so that the $3$-times enlarged balls in  $ \mathcal{F}^\lambda$ cover $\gamma$. More precisely,
 if we denote by $I$ the set of centers of the balls in $\mathcal{F}^\lambda$, then we have
 \begin{displaymath}
 B_{\lambda/3}(p)\cap B_{\lambda/3}(q) = \emptyset\,\,\,\,\text{ for  }\,p,q\in I,\,p\neq q\quad\text{and}\quad \gamma \subset \bigcup_{p\in I} B_{\lambda}(p).
 \end{displaymath}
Since $10k B_{\lambda}(p) \subset \Omega$ for all $p \in I$, we can apply Proposition \ref{p:ball_dist} (with $c=2$), and we have that
    \begin{align}
    \mathrm{diam}(f\circ \gamma) \leq \sum_{p \in I} \mathrm{diam} (f( B_{\lambda}(p)))  \leq 2 \sum_{p \in I} d(f(p), \partial \Omega'). \label{firsteq}
    \end{align}

    Next we establish an estimate of $\int_{\gamma} a_f ds$ from below by a multiple of $\sum_{p \in I} d(f(p), \partial \Omega')$. In what follows we employ the family of balls
    \begin{align}
    \{ B_{\lambda/3}(p) : p \in I \}. \label{korballsS}
    \end{align}
    %We recall that these balls are disjoint.
    Note that the family~\eqref{korballsS} does not necessarily cover $\gamma$, however since we are looking for a lower estimate of $\int_{\gamma} a_f \;\mathrm{d}s$, and $a_f \geq 0$, this will not be a problem. Let $ l(p)=\mathrm{length} ( \gamma \cap  B_{\lambda/3}(p) )$, then we claim that
    \begin{align}
    l(p) \geq  \frac{\lambda}{3} d(p, \partial \Omega). \label{blahblah3}
    \end{align}
    %DELETED:
% We note that
 %ADDED
 The bound
 \eqref{blahblah3} is obvious if $\gamma$ exits $
B_{\frac{\lambda}{3}}(p)$, however the assumption
\eqref{eq:curve_length_cond} implies  that \eqref{blahblah3} is
valid even if the entire curve is contained in $ B_{\frac{\lambda}{3}}(p)$.  Indeed, if $\gamma \subset  B_{\frac{\lambda}{3}}(p)$, then
    \begin{equation}\label{blahblah1}
    \alpha d(\gamma,\partial \Omega) \geq \alpha\left(d(p,\partial \Omega)- \frac{\lambda}{3} d(p,\partial \Omega)\right)\geq \frac{\lambda}{3} d(p,\partial \Omega)
    \end{equation} where the second inequality is a consequence of the choice of $\lambda$ as in \eqref{eq:condLambdaDiam}.  The fact that $l(p)=\mathrm{length}(\gamma) \geq \alpha d(\gamma,\partial \Omega)$, together with \eqref{blahblah1} proves \eqref{blahblah3} in this case.

    Our desired estimate will result by approximating $$\int_{\gamma \cap B_{\lambda/3}(p)} a_f ds$$  by $a_f(p)l(p)$, thus we require a constant $\tau$, depending on $K$ and the chosen metric only, such that  $a_f(x) \geq \tau d(f(p),\partial \Omega')/d(p,\partial \Omega)$ for all $ x \in B_{\lambda /3}(p)$. To this end we observe that Proposition \ref{p:ball_dist}, applied to  $B:=B_{\lambda/3}(p)$ and $c=2$, implies
    \begin{align*}
    d(f(x), \partial \Omega') & \geq d (  f( B_{\lambda/3}(p)  ), \partial \Omega' ) \geq  \frac{1}{2}   d(f(p), \partial \Omega')
    \end{align*} for all $ x \in B_{\lambda /3}(p)$.
% It follows that
Moreover,
    \begin{align*}
    d(x, \partial \Omega) \leq d(x,p) + d(p, \partial \Omega)\leq  \left(\frac{\lambda}{3} +1\right) d(p, \partial \Omega) \leq  2 d(p, \partial \Omega)
    \end{align*}
    for all $ x \in B_{\lambda/3}(p)$.
    Using the two inequalities above, together with Theorem~\ref{t:KoebeHeis}, we have
    \begin{align}\label{eq:a_f_compare}
    a_f(x) & \geq \frac{1}{c_K}\frac{ d(f(x), \partial \Omega')}{d(x, \partial \Omega)}  \geq \frac{1}{4c_K} \, \frac{ d(f(p), \partial \Omega')}{d(p, \partial \Omega)}
    \end{align}
    for all $ x \in B_{\lambda /3}(p)$. Thus $\tau=\frac{1}{4c_K} $ is sufficient for our needs.

    To finish the proof we observe that
    \begin{align*}
    \int_{\gamma} a_f ds  \geq \sum_{p \in I} \int_{\gamma \cap B_{\lambda /3}(p)} a_f ds
    \geq {\tau} \sum_{p \in I}    \frac{ d(f(p), \partial \Omega')}{d(p, \partial \Omega)}  l(p)
    & \overset{\eqref{blahblah3}}{\geq} {\tau} \frac{\lambda}{3} \sum_{p \in I}  d(f(p), \partial \Omega').
    \end{align*}
    This estimate combined with \eqref{firsteq} proves the claim with
    $ C = 6/{\lambda \tau}= {24 c_K}/{\lambda}.$
\end{proof}

\begin{remark}\label{r:example}
%DELETED:
% Analogously as in Euclidean spaces we can show that some assumption on the diameter of the curve $\gamma$ is necessary in Proposition \ref{t:diam}.
%ADDED:
Analogously as in Euclidean spaces, Proposition \ref{t:diam} does not hold without the assumption  \eqref{eq:curve_length_cond} on the length of $\gamma$. To see this,
  we consider the Heisenberg radial stretch map $f=f_k:\mathbb{H}^1 \to \mathbb{H}^1$, $0<k<1$, discussed in \cite{zbMATH06195912}. This is a quasiconformal mapping  which on the $(x,y)$-plane agrees with the usual planar radial stretch map, that is $f(z,0)=(z|z|^{k-1},0)$, and which sends points with Kor\'{a}nyi norm equal to $r\geq 0$ onto points of Kor\'{a}nyi norm $r^k$.
In light of Proposition \ref{t:diam}, let us now consider the map $f$ for $k=1/2$, restricted to the Kor\'{a}nyi unit ball, $\Omega= B(0,1)$ and let $\gamma$ be a line segment with $\mathrm{length}(\gamma)=r\in (0,1)$ on the $x$-axis emanating from $0$ (note that restricted to the $x$-axis, the Kor\'{a}nyi distance agrees with the Euclidean distance). Then $f(\gamma)$ is again a line segment on the $x$-axis starting from $0$, but with $\mathrm{diam}(f \circ \gamma)= \sqrt{r}$. For a fixed $\alpha$, we can choose $r>0$ small enough so that $\gamma$ violates the assumption \eqref{eq:curve_length_cond}, and by letting $r$ tend to $0$, we will see that indeed the conclusion of Proposition \ref{t:diam} does not hold in this case since $\sqrt{r}\gg r$ for small $r$, yet $\int_{\gamma} a_f\;\mathrm{d}s \leq cr$ for a positive and finite constant $c$ which does not depend on $r$. To establish the last claim it suffices to
observe that $a_f(0)<\infty$ and that there exists $r_0>0$ such that $a_f(x)\leq c' a_f(0)$ for all $x\in B(0,r_0)$ for a constant $c'$ depending on $K$, and in particular on $c_K$ from Theorem~\ref{t:KoebeHeis}. This can be seen by an argument as in \eqref{eq:a_f_compare}. Therefore,
$$
\int_{\gamma} a_f\;\mathrm{d}s \leq c' a_f(0) r \ll \sqrt{r}= \mathrm{diam}(f \circ \gamma),
$$
for $r<r_0$ small enough, which is impossible.
\end{remark}

\subsection{Comparison of the average derivative and the operator norm}\label{subs:comp}

As an application of the Koebe theorem for quasiconformal mappings in $\mathbb{R}^n$, Astala and Koskela have shown that for a $K$-quasiconformal map $f:\Omega  \to \Omega'$, for $\Omega,\Omega'\subseteq \mathbb{R}^n$, the integrals 
\begin{displaymath}
\int_{\Omega} \|Df(x)\|^p\;\mathrm{d}\mathcal{L}^n(x)\quad\text{and}\quad \int_{\Omega} a_f(x)^p\;\mathrm{d}\mathcal{L}^n(x).
\end{displaymath}
are quantitatively comparable for $p$ in an appropriate  parameter range, see Theorem 3.4 in \cite{MR1191747}.
The main goal of this section is to prove a counterpart of the aforementioned theorem, which is valid for both $d=d_s$ and $d=d_{\mathbb{H}^1}$.

\begin{thm}\label{t:comparable}
Let $f:\Omega \to \Omega'$ be a $K$-quasiconformal mapping between domains in $(\mathbb{H}^1,d)$ for some $K\geq 1$.
Moreover, denote by $p=p(K)>4$ a higher integrability exponent of the Jacobian of $f$ as in Theorem~\ref{t:GehringHeisenberg}. Then
\begin{displaymath}
 \frac{1}{c} \int_{\Omega} a_f(x)^q\;\mathrm{d}m \leq \int_{\Omega} \|D_H f(x)\|^q\;\mathrm{d}m \leq c \int_{\Omega} a_f(x)^q\;\mathrm{d}m
\end{displaymath}
for all $4-p < q < p$, where $c$ depends on $K$ and $q$.
\end{thm}

Theorem \ref{t:comparable} provides explicit bounds for the admissible exponents $q$ by using the exponent $p$ from Theorem~\ref{t:GehringHeisenberg} and Proposition~\ref{p:A_p_Jacobian}.
%DELETED:
% In particular, we obtain the lower bound $q\geq -\epsilon$ for $\epsilon=p-4$. Let us compare this to the expression $$\epsilon=\min\{p_0-n, 1/(p_1-1)\}$$ which is obtained for the $\mathbb{R}^n$-counterpart of this statement in \cite[Theorem 3.4]{MR1191747}. In this formula, $p_0$ denotes the exponent in Gehring's reverse H\"older's inequality for $\|Df\|$, which corresponds to our parameter $p$. The parameter $p_1$ is obtained from an estimate which is quantitatively equivalent to the Muckenhoupt weight condition for $\|Df\|^n$.
% (equivalently, for $J_f$).
%ADDED:
The $\mathbb{R}^n$-counterpart of this statement is \cite[Theorem 3.4]{MR1191747}, which gives the bounds
$$
-\min\{p_0-n, 1/(p_1-1)\}\leq q \leq n + \min\{p_0-n, 1/(p_1-1)\}.
$$
In this formula, $p_0$ denotes the exponent in Gehring's reverse H\"older's inequality for $\|Df\|$, which corresponds to our parameter $p$, and the parameter $p_1$ is obtained from an estimate which is quantitatively equivalent to the Muckenhoupt weight condition for $\|Df\|^n$ on a cube. Instead of an estimate in the spirit of this latter result, we apply Proposition~\ref{p:A_p_Jacobian}, which provides us with the bounds in terms of $p$.

Theorem \ref{t:comparable} is a consequence of our Koebe theorem (Theorem~\ref{t:KoebeHeis}), auxiliary results from Section \ref{s:JacobianQC} as well as the following observation, which is of independent interest.
%DELETED:
%In the proof of Theorem \ref{t:comparable}
%we employ a number of auxiliary results.
% The following key observation is of the independent interest.
Namely, it shows that $a_f$ as function of a point in the domain, satisfies a Harnack-type inequality. A similar estimate has already appeared in the proof of Proposition \ref{t:diam}, but in the following we need a more general statement which we formulate as follows. Let us also remark that this result can be deduced from our Koebe theorem, Proposition 3.7.3 in~\cite{LRj} and an appropriate localization argument. Instead, we give a short direct proof using just the Koebe theorem and Proposition \ref{p:ball_dist}.

\begin{lemma}\label{lem:afHI} The following holds with respect to $d\in \{d_s,d_{\Hei}\}$.
 Let $f:\Om \to \Om'$ be a $K$-quasiconformal map between domains $\Omega,\Omega' \subsetneq (\Hei,d)$.
 Suppose a ball $B\subset \Om$ satisfies the condition
 \begin{equation}\label{lem:whitney}
  \diam B\leq \lambda \dist(B, \partial \Om),
 \end{equation}
 where $0<\lambda \leq 2/(10k)$ for the constant $k$ from Proposition~\ref{p:ball_dist} applied to $c>1$ and the metric $d$. Then it holds
 \[
  a_f(x)\leq C a_f(y)
 \]
 for all $x,y\in B$ and a constant $C>1$ which depends only $c$ and $K$.
 \end{lemma}

\begin{proof}
 Let $B\subset \Om$ be as above and let us fix $x,y\in B$. Then, the Koebe theorem, see Theorem~\ref{t:KoebeHeis}, implies
 \begin{align}\label{ineq:afHI1}
  a_f(x)\leq c_K\,\frac{d(f(x), \partial \Om')}{d(x,\partial \Om)}&= c_K\,\frac{d(f(y), \partial \Om')}{d(y,\partial \Om)}\,\frac{d(f(x), \partial \Om')}{d(f(y),\partial \Om')}\,\frac{d(y, \partial \Om)}{d(x, \partial \Om)}\\&\leq c_K^2 a_f(y) \,\frac{d(f(x), \partial \Om')}{d(f(y),\partial \Om')}\,\frac{d(y, \partial \Om)}{d(x, \partial \Om)},\notag
 \end{align}
so the task is reduced to bounding the quotient in the last expression.

Letting $z$  be the center of the ball $B$, we have that
  \begin{align*}
 d(f(x), \partial \Om')&\leq  d(f(x), f(z))+d(f(z), \partial \Om')\\
  & \leq \diam f(B)+d(f(z), \partial \Om') \\
  & \leq \left(\tfrac{2}{c-1}+1\right)d(f(z), \partial \Om')\leq C' \dist(f(B), \partial \Om')
 \end{align*}
 with $C'=\left(\tfrac{2}{c-1}+1\right) \frac{c}{c-1}$. This follows from the last assertion of Proposition~\ref{p:ball_dist} applied with $c$. To justify the application of this proposition, we have to verify that $10kB \subset \Omega$ for the constant $k$ associated to
$c$. This is indeed the case since the choice of $\lambda$ ensures that
\begin{displaymath}
10 k \mathrm{rad}(B) = (10k)\frac{\mathrm{diam}B}{2} \leq \mathrm{dist}(B,\partial \Omega) \leq d(z,\partial \Omega).
\end{displaymath}
Therefore, the relevant term in \eqref{ineq:afHI1} can be bounded as follows:
 \begin{align*}
  \frac{d(f(x), \partial \Om')}{d(f(y),\partial \Om')}\,\frac{d(y, \partial \Om)}{d(x, \partial \Om)} &\leq
  2C'\frac{\dist(f(B), \partial \Om')}{d(f(y),\partial \Om')}\,\frac{d(y, \partial \Om)}{d(x, \partial \Om)}  \nonumber \\
  & \leq 2C'\,\frac{d(y, \partial \Om)}{d(x, \partial \Om)}.
 \end{align*}
To continue, we observe that \eqref{lem:whitney} yields
 \begin{displaymath}
 \frac{d(y,\partial \Omega)}{d(x,\partial \Omega)}\leq \frac{\mathrm{diam} B + \mathrm{dist}(B,\partial \Omega)}{\mathrm{dist}(B,\partial \Omega)}\leq \lambda +1,
 \end{displaymath}
 where the upper bound depends on $c$ and $K$ via the choice of $\lambda$. Thus we can find a constant $1<C<\infty$ such that \eqref{ineq:afHI1} reduces to $a_f(x)\leq C a_f(y)$, as desired.
\end{proof}

 %DELETED:
 % We now move to proving the main result of this section.
% The proof proceeds similarly to the corresponding result for quasiconformal mappings in the Euclidean setting, see Theorem 3.4 in \cite{MR1191747}. However, the proof in the Heisenberg setting involves several auxiliary observations for $\Hei$, proved above, such as Theorem~\ref{t:GehringHeisenberg}, Proposition~\ref{p:Lemma4}, Corollary~\ref{CorKoskAnalog}, Proposition~\ref{p:A_p_Jacobian}, Lemma~\ref{lem:afHI} and the Koebe theorem (Theorem~\ref{t:KoebeHeis}).

%ADDED:

Towards the proof of Theorem \ref{t:comparable}, we have to carefully choose a Whitney decomposition of our domain in order to ensure that the relevant balls are small enough so that all the necessary auxiliary results  from Section \ref{s:JacobianQC} and Lemma \ref{lem:afHI} are applicable. The results existing in the literature, cf.\ Proposition 4.1.15 in~\cite{MR3363168}, are not quite sufficient for our purpose since they state the existence of Whitney balls with a certain specific ratio between radii and distance to the complement of the domain. Adapting the proof in \cite{MR3363168}, we show the following result (both for the sub-Riemannian and the Kor\'{a}nyi distance) which leaves the flexibility to choose the parameter $\lambda$.

%The results existing so far in the literature are not fully covering our case, cf. Proposition 4.1.15 in~\cite{MR3363168}. Therefore, we need to show the following
%%DELETED:
%% Furthermore, the Whitney decomposition of the domain tailored for our construction must be carefully chosen in order to ensure all the restrictions imposed on balls in the course of the discussion of the aforementioned auxiliary results.
%result holding both for the sub-Riemannian and the Kor\'{a}nyi distance:

\begin{lemma}[Whitney decomposition]\label{l:whitney}
Let $\Omega\subsetneq \Hei$ be an open subset. For any $\lambda\in (0,1/2)$, there exists a countable collection
 $\mathcal{C}=\left\{B(x_i, r_i)\right\}$ of balls in $\Om$ such that
 \begin{align*}
  &(1)\quad \Om=\bigcup_{i} B(x_i,r_i)\\
  &(2)\quad \sum_{i} \chi_{B(x_i,2r_i)}\leq C,
 \end{align*}
 where the constant $C$ depends only on the choice of the metric $d\in \{d_s,d_{\Hei}\}$; and such that it holds
  \[
(3)\qquad  \frac{\lambda}{4}\dist(B, \partial \Om) \leq \diam B\leq \lambda \dist(B, \partial \Om),
 \]
 for any ball $B=B(x_i, r_i)$ in $\mathcal{C}$.
\end{lemma}
\begin{proof}
We fix a metric $d\in \{d_s,d_{\mathbb{H}^1}\}$.
Our goal is to find a collection of balls satisfying (1), (2), and
\[
(3')\qquad c_1(\lambda) d(x_i,\partial \Om)\leq  r_i\leq
c_2(\lambda) d(x_i,\partial \Om).
 \]
for
$
c_1(\lambda):= \frac{\lambda}{8}$ and $c_2(\lambda):= \frac{\lambda}{\lambda +2}$.
These constants have been chosen so that (3') implies (3). Indeed, assuming (3'), we find
\begin{align*}
\frac{\lambda}{4}\mathrm{dist}(B,\partial \Omega)\leq \frac{\lambda}{4} d(x_i,\partial \Omega) \leq 2 r_i = \mathrm{diam} B
\end{align*}
and
\begin{align*}
\mathrm{diam} B = 2 r_i \leq \frac{2\lambda}{\lambda +2} d(x_i,\partial \Omega)&\leq \frac{2\lambda}{\lambda +2} (\mathrm{dist}(B,\partial \Omega)+r_i)= \frac{2\lambda}{\lambda +2} (\mathrm{dist}(B,\partial \Omega)+ \tfrac{1}{2}\mathrm{diam}B),
\end{align*}
which implies the right-hand side of (3).
Thus it suffices to verify (1), (2) and (3'). We adapt the proof of Proposition 4.1.15 in \cite{MR3363168} and for any $k\in \mathbb{Z}$ define
 \[
 \mathcal{F}_k:=\left\{ B\left(x, \frac15 \frac{c_1(\lambda)+c_2(\lambda)}{2}d(x,\partial \Om)\right): x\in \Om\,\hbox{ and }\, 2^{k-1}\leq d(x,\partial \Om) \leq 2^k \right\}.
 \]
 We
apply the $5r$-covering lemma to find a countable pairwise
disjoint family of balls $\mathcal{G}_k\subset \mathcal{F}_k$ so that a family of balls
 \[
  \mathcal{C}:=\bigcup_{k\in \mathbb{Z}}\{5B:B\in \mathcal{G}_k\}
 \]
 satisfies assertion (1) of the lemma. By the definition of the radii $r_i$ as the arithmetic averages of $c_1(\lambda)d(x_i,\partial \Omega)$ and $c_2(\lambda)d(x_i,\partial \Omega)$, we get assertion (3'). In order to show (2) we proceed as in \cite{MR3363168}, exploiting the doubling property of the metric $d$. Suppose that there is $x\in \Om$ belonging to $M$ balls of the form $2B$, for $B\in \mathcal{C}$.
 We relabel the centers of these balls as $x_1,\ldots,x_M$ in such a way that $d(x_i,\partial \Omega)\leq d(x_1,\partial \Omega)$ for $i=1,\ldots,M$.
 The radii of the balls $2B_i$ are given by
 \[
 a(\lambda)d(x_i,\partial \Om):= (c_1(\lambda)+c_2(\lambda))d(x_i,\partial \Om).
 \]
 Note that the function $a(\cdot)$ is increasing on $[0,1/2]$. As $x$ lies in the intersection of the balls $2B_i$ centered at $x_i$, we find for all $i=1,\ldots,M$ that
\begin{displaymath}
d(x_1,x_i) \leq a(\lambda)(d(x_1,\partial \Om)+d(x_i,\partial \Om))
\end{displaymath}
and hence
\begin{displaymath}
d(x_i,\partial \Om)\geq d(x_1,\partial \Om) - d(x_1,x_i)\geq (1-a(\lambda)) d(x_1,\partial \Om)- a(\lambda)d(x_i,\partial \Om).
\end{displaymath}
This implies
\begin{displaymath}
d(x_i,\partial \Om) \geq \frac{1-a(\lambda)}{1+a(\lambda)}d(x_1,\partial \Om).
\end{displaymath}
 Moreover, for all $i$ we have $2B_i\subset B(x_1, 3R_1)$, with $3R_1=3a(\lambda)d(x_1, \partial \Om)$. If $x_i$ and $x_j$ are distinct centers of balls in the same family $\mathcal{G}_k$ we have, by disjointedness of the balls in $\mathcal{G}_k$, that
 \begin{displaymath}
 d(x_i,x_j)\geq \frac{1}{5}\frac{a(\lambda)}{2}\min\{d(x_i,\partial\Om),d(x_j,\partial\Om)\}\geq
  \frac{1}{5}\frac{a(\lambda)}{2}  \frac{1-a(\lambda)}{1+a(\lambda)}d(x_1,\partial \Om).
 \end{displaymath}
 That is, such points form a $\delta(\lambda)$-separated set for $$\delta(\lambda):= \frac{1}{5}\frac{a(\lambda)}{2}  \frac{1-a(\lambda)}{1+a(\lambda)}d(x_1, \partial \Om)$$ and they are all included in
  a ball of radius $R(\lambda):=3a(\lambda)d(x_1, \partial \Om)$. It is important for us to observe
 that $\delta(\lambda)/R(\lambda)$ can be bounded from below by a strictly positive number which does not depend on $\lambda$. This is the case since
 \begin{displaymath}
 \frac{1-a(\lambda)}{1+a(\lambda)}\geq \frac{1-a(1/2)}{1+a(1/2)}>0
 \end{displaymath}
 for all $\lambda\in (0,1/2)$. The doubling property (see Lemma 4.1.12 in \cite{MR3363168}) gives us that at most $N'$ of the balls $2B_i$ can have their centers in $\mathcal{F}_k$ for a fixed $k$, where $N'$ is a constant which depends only on the doubling constant associated to the metric $d$ and the universal lower bound for $\delta(\lambda)/R(\lambda)$. Next we show that the centers of the balls in our family $2B_1,\ldots,2B_M$ can lie in at most two different `layers' $\mathcal{F}_k$, which will provide the desired universal upper bound for $M$. Indeed, assume that $x_1\in\mathcal{F}_{k_1}$. Then, for every $i\in\{1,\ldots,M\}$, we find
 \begin{displaymath}
 d(x_1,\partial \Omega) \geq d(x_i,\partial \Omega) \geq \frac{1-a(\lambda)}{1+a(\lambda)} d(x_1,\partial \Omega)\geq \frac{1-a(1/2)}{1+a(1/2)} d(x_1,\partial \Omega)> \frac{1}{2}d(x_1,\partial \Omega).
 \end{displaymath}
The last estimate finally explains our choice of the bound $\lambda< 1/2$. This estimate shows
that all centers are contained in $\mathcal{F}_{k_1-1}\cup \mathcal{F}_{k_1}$ for some $k_1\in\mathbb{Z}$.
\end{proof}

\begin{proof}[Proof of Theorem~\ref{t:comparable}] Throughout the proof we work with the metric $d=d_{\mathbb{H}^1}$; the corresponding result for $d_s$ can be deduced from the final statement for $d_{\mathbb{H}^1}$.
Let $\lambda\in (0,\frac{1}{2})$ be the largest number for which the following conditions are satisfied:
\begin{enumerate}
\item  $\lambda \leq 2/(10k)$ where $k$ is\footnote{The exact value of $c$ is not essential, any constant larger than $1$ which depends at most on $K$ would work.}  as in Proposition~\ref{p:ball_dist} applied to $f$ and $c=2$,
\item $\lambda \leq 2/c$, where $c$ is as in Theorem \ref{t:GehringHeisenberg} applied to $f$,
\item $\lambda \leq 2/(10k)$, where $k$ is as in Proposition \ref{p:A_p_Jacobian} applied to $f$.
\end{enumerate}
These conditions are such that $\lambda$ is a positive constant depending on $K$, and every ball $B:=B(x_0,r)\subset \Omega$ with $\mathrm{diam} B \leq \lambda \mathrm{dist}(B,\partial \Omega)$ satisfies the assumptions of the following results (applied to the map $f$):
\begin{enumerate}
\item Proposition~\ref{p:ball_dist} (ball distortion) and Lemma~\ref{lem:afHI} (Harnack-type inequality for $a_f$) for $c=2$,
\item  Theorem~\ref{t:GehringHeisenberg} (higher integrability),
%Theorem~\ref{t:KoebeHeis} (the Koebe theorem),
\item Proposition~\ref{p:A_p_Jacobian} (weight property of the Jacobian).
\end{enumerate}
We will prove a statement for such balls which in addition satisfy a lower bound on the diameter:
 \begin{equation}\label{whitney-decomp}
\frac{\lambda}{4}\dist(B, \partial \Om) \leq \diam B\leq \lambda \dist(B, \partial \Om).
 \end{equation}
 %ADDED:
 Following the approach in the proof of \cite[Theorem 3.4]{MR1191747},
 our first step is to obtain a double inequality comparing $a_f(x_0)$ to a mean value of the appropriate power of $\|D_Hf\|$ over the ball $B$.

 By quasiconformality of $f$, condition \eqref{whitney-decomp}, Proposition~\ref{p:ball_dist} applied to $B$, $f$ and $c=2$, and Theorem~\ref{t:KoebeHeis}, we obtain that
 \begin{align*}
  \Barint_{B} \|D_Hf\|^4 \;\mathrm{d}m\leq K \frac{|f(B)|}{|B|} = K \frac{\mathrm{diam}(f(B))^4}{\mathrm{diam}(B)^4}&\leq K \frac{(2 d(f(x_0),\partial \Omega'))^4}{(\frac{\lambda}{4}\mathrm{dist}(B,\partial \Omega))^4}\\
  &\leq K \frac{(2 d(f(x_0),\partial \Omega'))^4}{(\frac{\lambda}{2(\lambda+2)}d(x_0,\partial \Omega))^4}\\
  &\leq \frac{K  4^4 (\lambda +2)^4 c_K^4}{\lambda^4} a_f^4(x_0).
 \end{align*}
 Thus, denoting $C'(K):= \frac{K^{1/4} 4 (\lambda +2) c_K}{\lambda} $, we find that
  \begin{equation}
  \Barint_{B} \|D_Hf\|^4 \;\mathrm{d}m\leq C'(K)^4 a_f^4(x_0). \label{est:t-compa2}
 \end{equation}
 We emphasize that $C'(K)$ depends on $K$ only (as $\lambda$ was chosen depending only on $K$, see the discussion in the beginning of the proof).

 In order to obtain a similar estimate for $a_f(x_0)$ from above, we appeal to a reasoning similar to the proof of Lemma~\ref{l:a_f_comparison}. Namely, let $B_1:= B(x_0,d(x_0,\partial \Omega))$. Recall that by the discussion in Section~\ref{s:JacobianQC} we have that $\log J_f \in \mathrm{BMO}(\Omega)$ and $\|\log J_f\|_{\ast}$ can be bounded in terms of $K$. Consequently
\begin{equation}
|(\log J_f)_{B_1}-(\log J_f)_{B}|\leq C \|\log J_f\|_{\ast},
\end{equation}
for a constant $C$ depending on $K$ via the ratio
$$
\frac{|B_1|}{|B|}=\frac{d(x_0,\partial \Omega)^4}{r^4}\leq\frac{(r+\dist(B,\partial \Om))^4}{r^4}\leq \left(1+\frac{8}{\lambda}\right)^4,
$$
see \eqref{whitney-decomp}. Therefore, since $\|\log J_f\|_{\ast}$ can be bounded by a constant in terms of $K$, we have
\begin{align}
a_f(x_0)^4&=\exp\left(\left(\log J_f\right)_{B_1}\right) \leq \exp\left(\left(\log J_f\right)_{B}+C \|\log J_f\|_{\ast}\right) \nonumber \\
&\leq C(K) \Barint_{B} J_f \;\mathrm{d}m\leq C(K) \Barint_{B} \|D_Hf\|^4\;\mathrm{d}m \label{est:comp}
\end{align}
for a constant $C(K)$ which depends only on $K$.
Here, we have used the Jensen inequality for the convex function $e^t$ and the Hadamard inequality in order to estimate $J_f=(\det D_H f)^2$ in terms of $\|D_Hf\|^4$.

So far we have shown that $a_f(x_0)$ is comparable to the average of $\|D_Hf \|^4$ over $B=B(x_0,r)$. The next goal is to replace ``$4$'' by a different power.
Starting from \eqref{est:comp}, we apply the H\"older inequality, the Gehring-type estimate in Theorem~\ref{t:GehringHeisenberg} (with exponent $p>4$) together with Proposition~\ref{p:A_p_Jacobian} to arrive at the following estimates:
\begin{align}
 a_f(x_0)&\leq C(K) \left(\Barint_{B} \|D_Hf\|^4\;\mathrm{d}m\right)^{\frac{1}{4}}\leq C(K) \left(\Barint_{B} \|D_Hf\|^p\;\mathrm{d}m\right)^{\frac{1}{p}} \nonumber \\
 &\leq C(K) K^{1/4} \left(\Barint_{B} J_f^{\frac{p}{4}}\;\mathrm{d}m\right)^{\frac{1}{p}} \leq c(K)  \left(\Barint_{B} J_f\;\mathrm{d}m\right)^{\frac{1}{4}} \nonumber \\
 &\leq c(K) \left(\Barint_{B} J_f^{1-\frac{p}{4}}\;\mathrm{d}m\right)^{\frac{1}{4-p}}\leq  c(K) \left(\Barint_{B} \|D_Hf\|^{4-p} \;\mathrm{d}m \right)^{\frac{1}{4-p}}. \label{est:t-compa1}
\end{align}

 As in the proof of \cite[Theorem 3.4]{MR1191747} we recall the following inequality for $g\in L^1(B)$ and $\epsilon>0$, whose proof is a direct consequence of the H\"older and the Jensen inequalities:
 \begin{equation}
\left(\Barint_{B} \frac{1}{|g|^\epsilon}\;\mathrm{d}m\right)^{-\frac{1}{\epsilon}}\leq \Barint_{B} |g|\;\mathrm{d}m. \label{ineq:neg-holder}
 \end{equation}
 This estimate applied for $g:=\|D_Hf\|$ and $\epsilon:=p-4$, together with the H\"older inequality, gives the following:
 \begin{equation}
   \left(\Barint_{B} \|D_Hf\|^{4-p}\;\mathrm{d}m \right)^{\frac{1}{4-p}}\leq \Barint_{B} \|D_Hf\| \;\mathrm{d}m  \leq \left(\Barint_{B} \|D_Hf\|^4 \;\mathrm{d}m \right)^{\frac{1}{4}}.
 \end{equation}
 This combined with \eqref{est:t-compa2}  results in a lower integral estimate for $a_f(x_0)$ in terms of $\|D_Hf\|^{4-p}$:
 \begin{equation}
   \left(\Barint_{B} \|D_Hf\|^{4-p} \;\mathrm{d}m  \right)^{\frac{1}{4-p}}\leq C'(K)\, a_f(x_0).\label{est:t-compa3}
 \end{equation}
 At this stage we apply Lemma~\ref{lem:afHI} (a Harnack-type inequality) together with estimate~\eqref{est:t-compa1} (for $0<q<p$) or
  estimate~\eqref{est:t-compa3} (for $4-p<q<0$) and obtain that
 \[
  \Barint_{B} a_f^q\;\mathrm{d}m\leq C^q \Barint_{B} a_f(x_0)^q\;\mathrm{d}m = C^q a_f(x_0)^q \leq C(K) \left(\Barint_{B} \|D_Hf\|^{4-p} \;\mathrm{d}m\right)^{\frac{q}{4-p}},
 \]
 where the constant $C(K)$ arises as a product of $q$-th powers of the constants $C$ in Lemma~\ref{lem:afHI} (for $c=2$) and $c(K)$ in \eqref{est:t-compa1} (or $C'(K)$ \eqref{est:t-compa3}, depending on the sign of $q$). We wish to estimate the above integral further from above.

 We consider three cases: (1) $4-p<q<0$, (2) $0<q<1$, (3) $1\leq q <p$. In the first case, the H\"older inequality gives us that
 \[
  \left(\Barint_{B} \frac{\mathrm{d}m}{\|D_Hf\|^{-q}}\right)^{-\frac{1}{q}} \leq  \left(\Barint_{B} \frac{\mathrm{d}m}{\|D_Hf\|^{(-q)\frac{p-4}{-q}}}\right)^{\frac{1}{p-4}}.
 \]
 In the second case a direct application of \eqref{ineq:neg-holder} for $g:=\|D_Hf\|^q$ and $\epsilon:=(p-4)/q$ results in the following estimate:
 \begin{equation}
 \left(\Barint_{B} \|D_Hf\|^{4-p} \mathrm{d}m \right)^{\frac{q}{4-p}} \leq \Barint_{B} \|D_Hf\|^q \mathrm{d}m. \label{est:t-compa4}
 \end{equation}
 Finally, in the third case we apply \eqref{ineq:neg-holder} and the H\"older inequality to obtain the estimate~\eqref{est:t-compa4}. Therefore, as a consequence of the above case analysis, we get
 \[
  \Barint_{B} a_f^q \;\mathrm{d}m \leq c(K) \Barint_{B} \|D_Hf\|^q\;\mathrm{d}m.
 \]
 By the analogous estimates we obtain the lower bound for the mean value of $a_f^q$ over $B$. Thus, it holds that
 \[
  \frac{1}{c(K)}\Barint_{B} a_f^q \;\mathrm{d}m\leq \Barint_{B} \|D_Hf\|^q \;\mathrm{d}m \leq c(K) \Barint_{B} a_f^q \;\mathrm{d}m.
 \]
 In the last step we apply the Whitney decomposition argument and show that $\Om$ can be expressed as a union of balls with controlled overlap satisfying \eqref{whitney-decomp}. That this is indeed the case, follows from Lemma~\ref{l:whitney}.
\end{proof}

As in the Euclidean case we have the following consequence of Theorem~\ref{t:comparable}, cf. Corollary 3.5 in \cite{MR1191747}.

\begin{cor}
Let $f:\Omega \to \Omega'$ be a $K$-quasiconformal mapping between domains in $\mathbb{H}^1$ for some $K\geq 1$.
Then for all $4-p < q < p$ with $p=p(K)>4$ and $c$ depending on $K$ and $q$, it holds that
\begin{displaymath}
\frac{1}{c} \int_{\Omega'} a_{\finv}(x)^{4-q}\;\mathrm{d}\mu \leq \int_{\Omega} a_f(x)^q\;\mathrm{d}\mu \leq c \int_{\Omega'} a_{\finv}(x)^{4-q}\;\mathrm{d}\mu.
\end{displaymath}
\end{cor}
\begin{proof}
 The proof follows the same lines as the proof of the corresponding result, Corollary 3.5 in \cite{MR1191747}, and is based on the change of variable formula, see e.g. Theorem 5.4(a) in \cite{MR1778673} and the proof of Proposition~\ref{p:A_p_Jacobian}.
\end{proof}

\subsection{Quasiconformal metrics on domains in $\Hei$}\label{sec:metrics}

 In \cite{bkr98}, the authors study quasiconformal metrics (more precisely, densities) defined on the unit ball in $\R^n$.
 %DELETED:
%  The motivating example for their considerations is that a conformal map from the planar unit disc into $\C$ is, up to post-compositions with isometries, uniquely determined by the absolute value of its derivative. This suggests to think of the latter as a `density' on the unit disk, and it motivates the following axiomatic definition.
The terminology is motivated by the fact that a  conformal map from the planar unit disc into $\C$ is, up to post-compositions with isometries, uniquely determined by the absolute value of its derivative, and hence the latter can be thought of as a `density' on the unit disk.

 Let $\mathbb{B}=B(0,1)\subset \R^n$ be the unit ball in the Euclidean space $\mathbb{R}^n$. Let further $\varrho:\mathbb{B}\to (0,\infty)$ be a strictly positive continuous function (called a \emph{density}), satisfying the following conditions, cf. Section 1 in~\cite{bkr98}:
 \begin{itemize}
 \item[(1)] (Harnack-type inequality.)\, There exist constants  $\lambda \in (0,1)$ and $c\geq 1$ such that
 \[
  \frac{1}{c}\leq \frac{\varrho(x)}{\varrho(y)}\leq c,\quad \hbox{for all }x,y\in B(z,\lambda d(z,\partial \mathbb{B})) \quad \hbox{for }z\in \mathbb{B}.
 \]
%  The condition means that there exist a constant  \varrho(x)\approx \varrho(y)
 \item[(2)] (Upper Ahlfors regularity with respect to $d_{\varrho}$.)\, There exists a constant $A>0$ such that
 \begin{equation}
  \mu_{\varrho}(B_{\varrho}(x,r)):=\int_{B_{\varrho}(x,r)}\varrho^n(y)\,d\mathcal{L}^n(y) \leq Ar^n,\quad \hbox{for all }x\in\mathbb{B}, r>0. \label{p:densities-Ahlf}
 \end{equation}
 Here $B_{\varrho}(x,r)$ stands for an open ball with respect to the length metric
 %ADDED \subset\mathbb{B}
 $d_{\varrho}(a,b):=\inf_{\gamma\subset \mathbb{B}} l_{\varrho}(\gamma)$ with  weighted length $l_{\varrho}(\gamma)=\int_{\gamma}\varrho\,ds$ and locally rectifiable curves $\gamma$ joining  $a,b\in \mathbb{B}$.
 %DELETED
 % within $\mathbb{B}$.
 \end{itemize}

It turns out, see \cite{bkr98}, that these simple conditions imposed on a density function are enough to infer several interesting geometric properties of distances defined via such densities.
 Among the examples of such densities studied in \cite[Section 2.4]{bkr98}, is
 \[
 \varrho:=\left(\Barint_{B(x,\dist(x,\partial \mathbb{B}))} \,J_f\,d\mathcal{L}^n\right)^{\frac{1}{n}},
 \]
 where $f:\mathbb{B}^n\to \Om$ is a $K$-quasiconformal mapping from the unit ball in $\R^n$ into a domain $\Om\subset \R^n$. The purpose of this section is to show a counterpart of this observation for  quasiconformal mappings between domains in $\Hei$.
The results of this paper allow us to move beyond the setting of mappings from a unit ball and study more general domains in $\Hei$.
%DELETED:
%Since the sub-Riemannian distance $d_s$ is the length distance associated to the Kor\'{a}nyi metric $d_{\mathbb{H}^1}$, we obtain the same distance $d_{\varrho}$ with respect to either metric for any density $\varrho$.
The following holds both for $d=d_s$ and $d=d_{\mathbb{H}^1}$, and the length element $ds$ in the definition of $l_{\varrho}$ taken with respect to the distance $d$:
 \begin{proposition}\label{p:a_f_conf_density}
  Let $f:\Om \to \Om'$ be a $K$-quasiconformal map between domains $\Om,\Om' \subsetneq\Hei$. Then the function $a_f$ possesses the following properties:
 \begin{itemize}
 \item[(1)] There exists a constant $\lambda \in (0,1)$ such that for all balls $B\subset \Om$ satisfying $\mathrm{diam} B \leq \lambda \mathrm{dist}(B,\partial \Om)$, it  holds
 \[
   \frac{1}{M} a_f(x)\leq a_f(y) \leq M a_f(x)\quad \hbox{for all } x,y\in B,
  \]
  with the equivalence constant $M$ depending on $K$ and the choice of $d\in \{d_s,d_{\Hei}\}$.
 \item[(2)] The upper Ahlfors regularity holds for the measure $\mu_{\varrho}$ as in \eqref{p:densities-Ahlf} with
 $\varrho=a_f$, $n=4$ and constants depending on $K$ and the choice of $d\in \{d_s,d_{\Hei}\}$.
  \end{itemize}
 \end{proposition}
 \begin{proof}
  Assertion (1) follows from Lemma~\ref{lem:afHI} applied to a fixed universal constant $c$.

  In order to prove the second assertion for $d=d_s$ (and a posteriori for $d=d_{\mathbb{H}^1}$) we follow the steps of the proof of the corresponding property for quasiconformal mappings from a unit ball in $\R^n$ into $\R^n$, see \cite[Section 2.4]{bkr98}. Let $x\in \Om$ and $r>0$. We consider two cases.

  \medskip

 {\sc Case 1}: Suppose that $r\leq c(K) a_f(x)d_s(x,\partial \Om)$ for a constant $0<c(K)<1$ depending only on $K$ and to be determined later. Let $\lambda:=2/(10k)>0$ be the constant from Lemma \ref{lem:afHI} associated to, say, $c=2$. We have the following inclusion of sub-Riemannian balls
  \begin{displaymath}
  B_s \left(x,\frac{\lambda r}{(\lambda+2) c(K) a_f(x)}\right) \subseteq B_s \left(x,\frac{\lambda}{\lambda+2}d_s(x,\partial \Omega)\right).
  \end{displaymath}
  Here the radius of the smaller ball has been chosen so that it is included in a ball which satisfies the assumption of Lemma \ref{lem:afHI}, so that a Harnack-type inequality for $a_f$ is valid on that ball. (Note that the constants given by Lemma \ref{lem:afHI} depend on the choice of the metric $d=d_s$).
  %DELETED:
  %Note that theng.

  Consider now $z\in B_{\varrho}(x,r)$. By definition,
  \begin{displaymath}
  d_{\varrho}(x,z)=\inf_{\gamma_{xz}} \int_{\gamma_{xz}} a_f ds = a_f(x) \inf_{\gamma_{xz}} \int_{\gamma_{xz}} \frac{a_f}{a_f(x)} ds<r
  \end{displaymath}
  where $\gamma_{xz}$ is a an arbitrary (locally rectifiable) curve  joining $x$ and $z$ within $\Omega$. The plan is to apply Lemma \ref{lem:afHI} in order to bound this quantity from below by $\frac{1}{C} a_f(x) d_s(x,z)$ for a positive and finite constant
   $C>1$, which depends only on $K$. To justify the application of  Lemma \ref{lem:afHI}, it suffices to ensure that we can consider curves $\gamma_{xz}$ which stay inside the sub-Riemannian ball $B_s(x,\frac{\lambda r}{(\lambda+2) c(K) a_f(x)})$. Let us explain why this is the case. First, since $z\in B_{\varrho}(x,r)$, there exists a rectifiable curve $\gamma_{xz}$ which connects $x$ to $z$ and satisfies
  \begin{equation}\label{eq:curve_less_r}
  \int_{\gamma_{xz}}a_f \;\mathrm{d}s<r.
  \end{equation}
  In the definition of $d_{\varrho}(x,z)$ we can restrict the infimum to curves satisfying \eqref{eq:curve_less_r}.
Assume that such a curve $\gamma_{xz}$ exits $B_s(x,\frac{\lambda r}{(\lambda+2) c(K) a_f(x)})$. Then, by connectedness, there must exist a (first) point $w$ on the trace of $\gamma_{xz}$ with
$$
w\in \partial B_s\left(x,\frac{\lambda r}{(\lambda+2) c(K) a_f(x)}\right).
$$
We denote by $\gamma_{xw}$ the subcurve of $\gamma_{xz}$ which connects $x$ and $w$ inside $B_s(x,\frac{\lambda r}{(\lambda+2) c(K) a_f(x)})$. Since $a_f$ is a positive function, we find
\begin{align*}
\int_{\gamma_{xz}}a_f\;\mathrm{d}s&\geq \int_{\gamma_{xw}}a_f\;\mathrm{d}s\geq \frac{1}{C} a_f(x) \int_{\gamma_{xw}}\;\mathrm{d}s\geq \frac{1}{C} a_f(x) d_s(x,w) = \frac{\lambda r}{(\lambda+2) \cdot  C \cdot c(K)}.
\end{align*}
We may choose $0<c(K)<1$ such that
\begin{equation}\label{eq:c(K)}
c(K)< \frac{\lambda}{(\lambda+2)C},
\end{equation}
which leads to a contradiction to the assumption $\int_{\gamma_{xz}}a_f\;\mathrm{d}s<r$. With this choice of $c(K)$, we may restrict the curves in the definition of $d_{\varrho}(x,z)$ to those curves $\gamma_{xz}$ along which the Harnack inequality for $a_f$ is valid, and we find
\begin{displaymath}
\frac{1}{C} a_f(x) d_s(x,z) \leq d_{\varrho}(x,z) < r.
\end{displaymath}
In particular we have for our choice of $c(K)$ that
\begin{displaymath}
B_{\varrho}(x,r) \subseteq B_s\left(x,C \frac{r}{a_f(x)}\right) \subseteq B_s \left(x,\frac{\lambda}{\lambda+2}d_s(x,\partial \Omega)\right),
\end{displaymath}
for $0<r\leq c_K a_f(x) d_s(x,\partial \Omega)$.
Since the Harnack inequality from Lemma \ref{lem:afHI} is valid on $B_{\varrho}(x,r)$, we find
  \[
   \mu_{\varrho}(B_{\varrho}(x,r))=\int_{B_{\varrho}(x,r)} a_f(y)^4\,\mathrm{d}y\leq C^4 a_f(x)^4 m\left(B_{\varrho}(x,r)\right).
  \]
Thus, we obtain:
  \begin{align*}
  %DELETED:
   \mu_{\varrho}(B_{\varrho}(x,r))
   %&\leq C^4 a_f(x)^4 m\left(B_{\varrho}(x,r)\right) \\ &
   \leq C^4 a_f(x)^4 m\left(B_s\left(x, C \frac{r}{a_f(x)}\right)\right) \leq C^8 r^4
  \end{align*}
  and the proposition is proven in this case.

  \medskip

{\sc Case 2}: Let us now consider the case $r\geq c(K) a_f(x)d_s(x,\partial \Om)$. Then, by Theorem~\ref{t:KoebeHeis} we have $r\geq (c(K)/c_K) d_s(f(x),\partial \Om')$. We will use this estimate below.

\medskip

 {\sc Step 1: the Whitney-type decomposition of $\Om$.} By Lemma~\ref{l:whitney} let us decompose $\Om$ as a union of balls satisfying the Whitney condition~\eqref{whitney-decomp} for $d_s$ and $\lambda>0$ the largest number, possibly different from the first part of the proof, for which the following conditions are satisfied
\begin{equation}\label{eq:condW}
\lambda < \frac{\lambda}{1-\lambda}\leq \frac15
\end{equation}
and
\begin{equation}\label{eq:condW2}
\lambda  \leq \frac{2}{10k}.
\end{equation}
The first condition is related to Proposition~\ref{p:egg_yolk} (egg yolk principle for $d_s$) and $k$ is as in Proposition~\ref{p:ball_dist} applied to $f$, $c=2$, and $d=d_s$. The value of $\lambda$ thus depends only on $K$ (and the metric $d_s$). The first condition, \eqref{eq:condW}, is to ensure quasisymmetry of $f$ on all the relevant balls which will appear later in the proof. The second condition, \eqref{eq:condW2}, is to guarantee that every ball in the constructed Whitney decomposition satisfies the assumptions of Lemma~\ref{lem:afHI} applied to the map $f$.
%DELETED:
%Since the balls in the respective metric satisfy $B_s(x,r) \subseteq B_{\mathbb{H}^1}(x,\sqrt{\pi}r)$, this ensures that we can from now on work with the sub-Riemannian distance $d_s$, which we denote for simplicity by $d$.

\medskip

 {\sc Step 2.} Let $\mathcal{C}_x$ be the collection of those sub-Riemannian balls $B$ in the chosen Whitney  decomposition for which
  $B\cap B_{\varrho}(x,r)\not=\emptyset$. Then, we claim that
  \begin{equation}
   f\left(\bigcup_{B\in \mathcal{C}_x}B\right)\subseteq B_{s}(f(x),
   cr),\label{p:condBKR}
  \end{equation}
  for some constant $c>0$, which can be bounded from above in terms of $K$.

  In order to show \eqref{p:condBKR}, let us consider $y\in B$ for $B\in \mathcal{C}_x$ and discuss separately the two
  cases: (i) $y\in B_{\varrho}(x,r)$ and (ii) $y\in B\setminus B_{\varrho}(x,r)$.
  In the first case, by the definition of $d_{\varrho}$ there exists a rectifiable curve
  $\gamma$ joining $x$ and $y$ with
  $l_{\varrho}(\gamma)=\int_{\gamma}a_f(s)ds<r$. Motivated by the egg yolk
  principle, Proposition~\ref{p:egg_yolk}, let
 $\alpha:=\frac{1}{5}$. If $d(x,y)\geq \alpha d(x, \partial \Om)$, then $\mathrm{length}(\gamma)\geq \alpha d(\gamma, \partial \Om)$ and so Proposition~\ref{t:diam} allows us to conclude the following estimate:
  \[
   d(f(x), f(y))\leq \diam f(\gamma)\leq C \int_{\gamma}a_f(s) ds
   <Cr.
  \]
 From this
 \begin{displaymath}
 f\bigg(B\cap B_{\varrho}(x,r)\cap \{y: d(x,y)\geq \alpha d(x,\partial \Omega)\}\bigg) \subseteq B_s(f(x),cr)
 \end{displaymath}
  with $c\geq C$ follows, which is a first step towards the proof of
 \eqref{p:condBKR}.

 If $d(x,y)< \alpha d(x, \partial \Om)$, we will invoke Proposition~\ref{p:egg_yolk}, which we may by our choice of $\alpha$.
 Applied to $f$ and $\Omega$, this shows that there is a constant $H$, depending only on $K$, such that $f$ is $H$-quasisymmetric when restricted to $B(x, \alpha d(x, \partial \Om))={B(x, \frac{d(x, \partial \Om)}{5})}$.

 For $t>0$ and $x_0  \in B(x, \frac{d(x, \partial \Om)}{5})$, set\footnote{In the notation from the beginning of Section~\ref{subs-2-3}, the above expressions correspond to $L_g(x_0,t)$ and $l_g(x_0,t)$ for $g=f|_{B(x,d(x,\partial \Om)/5)}$.}
 \begin{align*}
  L_f(x_0,t)&:=\sup_{\footnotesize{\{z\in B(x, \frac{d(x, \partial \Om)}{5}):\, d(x_0,z)\leq t\}}} d(f(x_0),
  f(z)),\\
  l_f(x_0,t)&:=\inf_{\{z\in B(x, \frac{d(x, \partial \Om)}{5}):\, d(x_0,z)\geq t\}} d(f(x_0),
  f(z)).
 \end{align*}
 With this notation, it holds that
  \begin{align*}
  d(f(x), f(y))
  \leq L_f(x, \alpha d(x, \partial \Om)) \leq H l_f(x, \alpha d(x, \partial \Om))
  \leq H d(f(x), \partial \Om') \leq H \frac{c_K}{c(K)}r.
 \end{align*}
 In the last step we use the assumption that $r\geq c(K)/c_K d(f(x),\partial
 \Om')$. Altogether we have shown that
 \begin{displaymath}
 f(B\cap B_{\varrho}(x,r)) \subseteq B_s(f(x),cr)
 \end{displaymath}
 holds with $c\geq \max\{C, H c_K/c(K)\}$. This concludes the discussion of
 %DELETED:
 %case (i) regarding \eqref{p:condBKR}.
%ADDED:
\eqref{p:condBKR} for case (i).

 For (ii), suppose that $y\in B \setminus B_{\varrho}(x,r)$ for some ball $B\in\mathcal{C}_x$. Then, by the definition of $\mathcal{C}_x$, there is $z\in B\cap  B_{\varrho}(x,r)$ and it holds that
 \[
  d(f(x), f(y))\leq  d(f(x), f(z))+  d(f(z), f(y)).
 \]
 The first term on the right-hand side above can be estimated by the reasoning of the previous case, since in particular $z\in B_{\varrho}(x,r)$. In order to estimate the second term, we proceed as follows. Let $x_B$ be the
 center of $B$. Then, by Whitney condition~\eqref{whitney-decomp} for $\lambda$ we observe that
 \[
  d(x_B,z)\leq \diam B\leq \lambda \dist(B, \partial \Om)\leq \lambda
  d(z, \partial \Om).
 \]
 Thus, $x_B\in B(z, \lambda d(z, \partial \Om))$. Using this observation together with the definition of Whitney-type decomposition~\eqref{whitney-decomp} with balls satisfying condition \eqref{eq:condW}, we see that the conclusion of the egg yolk principle holds on
 \begin{displaymath}
 B(x_B,\lambda d(x_B,\partial \Omega)) \supseteq B(x_B,\mathrm{diam}B) \supseteq B
 \end{displaymath}
 and on
 \begin{displaymath}
 B(x_B, d(x_B,\partial \Omega)/5))\supseteq B(z,\lambda d(z,\partial \Omega)).
 \end{displaymath}
 Thus, exploiting the quasisymmetry property of $f$ on the respective balls, we get by similar estimates as in the proof of \cite[Proposition 3.7.5]{LRj} that
 \begin{align*}
  d(f(z), f(y))&
  %\leq d(f(z), f(x_B))+  d(f(x_B), f(y))
  %\\
  %&\leq 2L_f(x_B, \diam B)\leq 2H l_f (x_B, \diam B)\\
 % &\leq 2H d(f(x_B), \partial \Om')\\
  %&\leq  2H(d(f(x_B), f(z))+d(f(z), \partial \Om')) \\
 % &\leq 2H(L_f(z, \lambda d(z, \partial \Om))+d(f(z), \partial \Om'))\leq  2H(l_f(z, \lambda d(z, \partial \Om))+d(f(z), \partial \Om'))\\
  %&\leq 2H(H+1) d(f(z), \partial \Om')
  \leq 2H(H+1) \left(d(f(z), f(x))+d(f(x), \partial \Om')\right)\\
  &\leq  2H(H+1)\left(\max\{C,H \tfrac{c_K}{c(K)}\}+c(K)/c_K\right)r.
 \end{align*}
 In the last step we appeal to the previously discussed case (as $z \in B_{\varrho}(x,r)$) and use the assumption that $r\geq c(K)/c_K d(f(x),\partial \Om')$. This completes the proof of this case and the whole claim \eqref{p:condBKR}, as well.

 \medskip

 {\sc Step 3: the upper Ahlfors regularity.} In order complete the proof of the proposition we observe that Lemma~\ref{lem:afHI} together with the Jensen inequality for the exponential function and Lemma~\ref{l:a_f_comparison}, applied to a suitable $L$ depending on $\lambda$, allow us to infer that for $B\in \mathcal{C}_x$ it holds that
 \begin{equation}
  \int_{B} a_f(y)^4 \;\mathrm{d}m(y) \leq C' \int_{B} J_f(y)\;\mathrm{d}m(y) \label{eq:last}
 \end{equation}
 for a suitable constant $C'\geq 1$ which depends only on $K$,
 %DELETED:
 %. (This works analogously as in the proof of Theorem~\ref{t:comparable}).
 %ADDED:
 analogously as in the proof of Theorem~\ref{t:comparable}.
  Therefore,
  %DELETED:
  \begin{align*}
   \mu_{\varrho}(B_{\varrho}(x,r))
   %=\int_{B_{\varrho}(x,r)}a_f(y)^4\;\mathrm{d}m(y)
   \leq C' \sum_{B\in\mathcal{C}_x}\int_B J_f(y)\;\mathrm{d}m(y)
   % &
   \leq C m\left(B_{s}(f(x), cr)\right)
   %\\&
   \leq C'' r^4,
  \end{align*}
 by \eqref{p:condBKR} and the controlled overlap in the Whitney decomposition. Here the constants $C$ and $C''$ depend only on $K$. This completes the proof of the second assertion.
 %DELETED:
 % and the proof of the proposition.
 \end{proof}

 Proposition \ref{p:a_f_conf_density} shows the upper Ahlfors regularity of $\mu_{\varrho}$. More can be said
 if $\Om\subset \Hei$ equipped with the sub-Riemannian distance $d_s$ is $L$-\emph{quasiconvex}, that is,
 %DELETED:
% under additional assumptions on $\Omega$. Let us recall that a domain $\Om\subset \Hei$ equipped with the sub-Riemannian distance $d_s$ is called $L$-\emph{quasiconvex}
%ADDED:
 if any two points $x,y\in \Om$ can be joined by a curve $\gamma$ such that its trace $|\gamma|$ is in $\Om$ and $\mathrm{length}(\gamma)\leq L d_s(x,y)$.
 %DELETED
 %Among examples of such domains let us mention uniform domains, a half space and compact $C^{1,1}$-domains. For further examples of quasiconvex domains in $\Hei$, see \cite{HLT} and the references therein.
For examples of quasiconvex domains in $\Hei$, see \cite{HLT} and references therein.

   \begin{proposition}\label{p:stronger}
   Let $f:\Omega \to \Omega'$ be a $K$-quasiconformal map from a quasiconvex domain $\Omega \neq \mathbb{H}^1$ onto a domain $\Omega' \neq \mathbb{H}^1$. Then, there exist constants $0<c_1<c_2<\infty$ and $0<c(K)<1$ such that for all $x\in \Omega$ and all $0<r<c(K) a_f(x) d_s(x,\partial \Omega)$ one has
   \begin{displaymath}
  c_1 r^4 \leq  \mu_{\varrho}(B(x,r))\leq c_2 r^4.
   \end{displaymath}
   \end{proposition}

   \begin{proof} Let us assume that $\Omega$ is $L$-quasiconvex for some constant $L\geq 1$.
   %DELETED:
  % The upper bound for $ \mu_{\varrho}(B(x,r))$ has already been established in the first part of the proof of Proposition \ref{p:a_f_conf_density} (and this holds in fact for any $r>0$). Let us therefore consider the lower bound.
   By Proposition \ref{p:a_f_conf_density}, it suffices to prove the lower bound for $\mu_{\varrho}(B(x,r))$.
   If $c(K)$ is chosen as in the proof of Proposition \ref{p:a_f_conf_density}, that is, as in \eqref{eq:c(K)}, then we know already that the Harnack inequality for $a_f$ holds on $B_s(x,r/(Ca_f(x)L ))$. Thus, for all points $z$ in this ball, we find
   \begin{align*}
   d_{\varrho}(x,z) &= \inf_{\gamma_{xz}\subset \Omega}\int_{\gamma_{xz}}a_f\;\mathrm{d}s\leq C a_f(x) \inf_{\gamma_{xz}\subset \Omega}\int_{\gamma_{xz}}\;\mathrm{d}s\leq C a_f(x) L d_s(x,z),
   \end{align*}
   where we have used in the last step the assumption that $\Omega$ is $L$-quasiconvex. The above estimate shows that
   \begin{displaymath}
   B_s \left(x,\frac{r}{C a_f(x) L }\right) \subseteq B_{\varrho}(x,r)
   \end{displaymath}
   and hence
   \begin{displaymath}
 \frac{m(B(0,1))r^4}{C^8 L^4}   = \frac{a_f(x)^4}{C^4 }m\left(B_s\left(x,\frac{r}{C a_f(x) L}\right)\right) \leq \int_{B_s(x,r/(C a_f(x) L))} a_f^4 \;\mathrm{d}m\leq \mu_{\varrho}(B_{\varrho}(x,r)),
   \end{displaymath}
   which concludes the proof.
   \end{proof}

\bibliographystyle{plain}
\bibliography{references}

\end{document}